\documentclass[11pt, reqno]{amsart}
\usepackage{amsthm}
\usepackage{amssymb}
\usepackage{amsfonts}
\usepackage{amsmath}
\usepackage{comment}
\usepackage{hyperref}
\newtheorem{theorem}{Theorem}[section]
\newtheorem{lemma}[theorem]{Lemma}

\newtheorem{corollary}[theorem]{Corollary}
\newtheorem{proposition}[theorem]{Proposition}

\numberwithin{equation}{section}

\DeclareMathOperator{\supp}{supp} 
\DeclareMathOperator{\diam}{diam} \DeclareMathOperator{\dist}{dist} \DeclareMathOperator{\s}{span}
\DeclareMathOperator{\SOT}{SOT} \DeclareMathOperator{\WOT}{WOT} 
  \DeclareMathOperator{\tr}{tr}
\newcommand{\C}{\ensuremath{\mathbb{C}^n}} \newcommand{\B}{\ensuremath{B_\alpha }}
 
\newcommand{\Fp}{\ensuremath{F_\alpha ^p }} 
\newcommand{\Ft}{\ensuremath{F_\alpha ^2 }} \newcommand{\Lt}{\ensuremath{L_\alpha ^2 }}
\newcommand{\Lp}{\ensuremath{L_\alpha ^p }} \newcommand{\Z}{\ensuremath{\mathbb{Z}^{2n}}}
\newcommand{\Fpphi}{\ensuremath{F_\phi ^p }} \newcommand{\N}{\ensuremath{\mathbb{N}}}
 \newcommand{\Fqphi}{\ensuremath{F_\phi ^q }} \newcommand{\CalZ}{\ensuremath{\mathcal{Z}}}
\newcommand{\Ftwophi}{\ensuremath{F_\phi ^2 }} \newcommand{\Tp}{\ensuremath{\mathcal{T}_\alpha ^p}}
\newcommand{\Fonephi}{\ensuremath{F_\phi ^1 }}  \newcommand{\Lonephi}{\ensuremath{L_\phi ^1 }}
\newcommand{\Finfphi}{\ensuremath{F_\phi ^\infty }} 
\newcommand{\Lpphi}{\ensuremath{L_\phi ^p }} \newcommand{\D}{\ensuremath{\mathbb{D}}}
\newcommand{\Lqphi}{\ensuremath{L_\phi ^q }}  \newcommand{\Ltwoa}{\ensuremath{A^2(\D)}}
\newcommand{\Ltwophi}{\ensuremath{L_\phi ^2 }} \newcommand{\Bn}{\ensuremath{\mathbb{B}_n}}
\newcommand{\Ltwoaphi}{\ensuremath{A_\phi ^2(\Bn)}}
\newcommand{\Lpaphi}{\ensuremath{A_\phi ^p (\Bn)}}
\newcommand{\Linfphi}{\ensuremath{L_\phi ^\infty }}
\newcommand{\incn}{\ensuremath{\int_{\C}}}
\newcommand{\Tpphi}{\ensuremath{\mathcal{T}_\phi ^p}}
\newcommand{\SLphi}{\ensuremath{\mathcal{S}\mathcal{L} (\phi)}}
\newcommand{\SL}{\ensuremath{\mathcal{S}\mathcal{L}(\alpha)}}
\newcommand{\F}{\ensuremath{\mathcal{F}}}
\newcommand{\Tt}{\ensuremath{\mathcal{T}_\alpha ^2}}
\newcommand{\Tk}{\ensuremath{{\widetilde{k}}}}
\newcommand{\Q}{\ensuremath{\mathcal{Q}}}
\newcommand{\INF}[1] {\ensuremath{| #1|_{\infty}}}

\begin{document}

\title[{Compactness of operators on generalized Fock spaces}]{{Compactness and essential norm properties of operators on generalized Fock spaces}}




\author[Joshua Isralowitz]{Joshua Isralowitz}
\address{Joshua Isralowitz \\ Department of mathematics and statistics \\  SUNY Albany \\  Albany, NY  \\ 12222, USA}
\email{jisralowitz@albany.edu}

\begin{abstract}
The purpose of this paper is to systematically study compactness and essential norm properties of operators on a very general class of weighted Fock spaces over $\C$.  In particular, we obtain rather strong necessary and sufficient conditions for a wide class of operators (which includes operators in the Toeplitz algebra generated by bounded symbols) to be compact and we obtain related estimates on the essential norm of such operators. Finally, we discuss interesting open problems related to our results. \end{abstract}



\maketitle


\maketitle

\section{Introduction}
For some $\alpha > 0$, let $F_\alpha ^p$ be the classical Fock space of entire functions on $\C$ such that $f(\cdot) e^{-\frac{\alpha}{2} |\cdot|} \in L^p(\C, dv)$ where $dv$ is the ordinary Lebesgue volume measure. Let $K(z, w) = e^{\frac{\alpha}{2} z \cdot \overline{w}}$ be the reproducing kernel of $\Ft$ and let $k_z(w) = K(z, w) / \sqrt{K(w, w)}$ be the normalized reproducing kernel of $\Ft$.  If $A$ is a bounded operator on $\Fp$ for $1 < p < \infty$, then let $B(A)$ be the bounded function on $\C$ defined by \begin{equation*} (B(A))(z) = \langle Ak_z, k_z \rangle_{\Ft}. \end{equation*}  It is well known (and very easy to prove, see \cite{Zhu}) that $k_z \rightharpoonup 0$ weakly in $\Fp$ as $|z| \rightarrow \infty$ if $1 < p < \infty$.  Thus, if $A$ is compact on $\Fp$ and $1 < p < \infty$, then an easy application of H\"{o}lder's inequality immediately tells us that $B(A)$ vanishes at infinity.

On the other hand, one can easily come up with examples of bounded operators on $\Fp$ (in fact even bounded Toeplitz operators on $\Ft$, see \cite{BC}) whose Berezin transform vanishes at infinity but that are nonetheless not compact. This immediately raises the question of when the Berezin transform of a bounded operator vanishing at infinity implies the compactness of this operator.

Define the Toeplitz operator $T_f$ on $\Fp$ with $f \in L^\infty(\C)$ by the usual formula $T_f = P M_f$ where $P$ is the orthogonal projection from $\Lt$ to $\Ft$ and $M_f$ is ``multiplication by $f$" (and note that $T_f$ is bounded on $\Fp$ when $1 < p < \infty$ since $P$ is bounded on $\Lt$.) Furthermore, given any class of measurable functions $X$ on $\C$, let $\Tp(X)$ be the $\Fp$ operator norm closure of the algebra generated by $\{T_f : f \in X\}$.   Then the following theorem was recently proved by W. Bauer and J. Isralowitz (see \cite{BI})

\begin{theorem} \label{BIThm} If $1 < p < \infty$ and $A \in \Tp(L^\infty(\C))$ then $B(A)$ vanishing at infinity implies that $A$ is compact.  Furthermore, any compact operator on $\Fp$ is in fact in fact in $\Tp(L^\infty(\C))$. \end{theorem}

Before we continue let us mention some history leading up to this theorem.  First, note that the sufficiency part of Theorem \ref{BIThm} was first proved by S. Axler and D. Zheng in the seminal paper \cite{AZ} for the classical Bergman space $L_a ^2 (\mathbb{D}, dA)$ setting in the special case where $A$ is in the algebra generated by $\{T_f : f \in L^\infty(\mathbb{D})\}$.  Furthermore,  note that this result was extended to the $\Ft$ setting by  M. Engl\v{i}s in \cite{E}.  On the other hand,  Theorem \ref{BIThm} was later proved for $L_a ^p(\mathbb{B}_n, dv)$ in its entirety when $1 < p < \infty$ by D. Su\'{a}rez in \cite{S} using vastly more technical and deeper techniques than the ones in \cite{AZ} (see also \cite{MSW} where this result is extended to the canonically weighted Bergman space $L_a ^p(\mathbb{B}_n, dv_\gamma)$).  Moreover, note that the proof of Theorem \ref{BIThm} in \cite{BI} largely uses \cite{S} as a blueprint, though (as usual) the details involved in extending these arguments to the Fock space setting are often highly nontrivial and thus require considerable work. Also, note that both \cite{BI} and \cite{S} contain (as largely byproducts of the techniques used to prove Theorem \ref{BIThm} and its Bergman space analogue) interesting essential norm estimates for both general operators and operators in  $\Tp(L^\infty(\C))$  and respectively $\mathcal{T} ^p (L^\infty(\Bn))$.

Interestingly, note that Theorem \ref{BIThm} (and remarkably its Bergman space version in \cite{S}) was given a vastly simplified proof by M. Mitkovski and B. Wick in \cite{MW} using completely different methods than those of \cite{BI,S}. On the other hand, J. Xia and D. Zheng in the recent paper \cite{XZ} introduced the class $\SL$ of ``sufficiently localized'' operators on $\Ft$ consisting of those operators $A$ on $\Ft$ where \begin{equation*}
|\langle Ak_z,k_w\rangle_{\Ft }| \lesssim \frac{1}{\left(1+|{z-w}|\right)^{2n + \delta}} \end{equation*} for some $\delta = \delta(A) > 0$ independent of $z, w \in \C$, which is a $*-$algebra of bounded operators on $\Ft$ that contains all Toeplitz operators with bounded symbols.  Furthermore, Theorem \ref{BIThm} was generalized in the $\Ft$ setting in  \cite{XZ} as follows \begin{theorem} \label{XZThm} If $A$ is in the $\Ft$ operator norm closure of $\SL$ then $A$ is compact on $\Ft$ if $B(A)$ vanishing at infinity. \end{theorem}  \noindent Also, note that Theorem \ref{XZThm} was proved by frame theoretic ideas that are vastly simpler than the ones in \cite{BI,S}.

The purpose of this paper is to try to extend Theorem \ref{XZThm} to a very wide class of exponentially weighted Fock spaces, and more generally to study essential norm properties of operators on these weighted Fock spaces.  More precisely, let $d^c =  \frac{i}{4} (\overline{\partial} - \partial)$ and let $d$ be the usual exterior derivative. Let $\phi \in C^2(\C)$ be a real valued function on $\C$  such that \begin{equation} c \omega_0 < d d^c \phi < C \omega_0 \label{PhiCond} \end{equation} holds uniformly pointwise on $\C$ for some positive constants $c$ and $C$ (in the sense of positive $(1, 1)$ forms) where $\omega_0 = d d^c |\cdot |^2$ is the standard Euclidean K\"{a}hler form.   For any $1 \leq p \leq \infty$ and any positive Borel measure $\nu$ on $\C$,  let $\Lpphi(\nu)$ be the space defined by  \begin{equation*}  \Lpphi(\nu) := \{f \text{ measurable on } \C \text{ s.t. }  f(\cdot)e^{-\phi(\cdot)}  \in L^p(\C, d\nu). \} \end{equation*} Furthermore, let $\Lpphi$ be the space $\Lpphi(dv)$  and let $\Fpphi$ be the so called ``generalized Fock space" defined by \begin{equation*}  \Fpphi := \{f \text{ entire on } \C \text{ s.t. }  f \in \Lpphi \}. \end{equation*}     Note that the spaces $\Fpphi$ appear naturally in the study of the $\overline{\partial}$ equation and sampling/interpolation theory and have also been studied by numerous authors (see \cite{BO,C,D,L,OS,SV} for example, and in particular, see \cite{SV} for an excellent overview of the basic linear space properties of $\Fpphi$.)

Fix some real valued $\phi$ satisfying (\ref{PhiCond}) and let $k_z$ be the normalized reproducing kernel of $\Ftwophi$.  Furthermore, here and throughout the rest of the paper we will let $\langle \cdot, \cdot \rangle$ denote the canonical $\Ftwophi$ inner product. For a bounded operator $A$ on $\Fpphi$ with $1 < p < \infty$, let $B(A)$ again be the Berezin transform of $A$ defined on $\C$ by \begin{equation*} (B(A))(z) : = \langle Ak_z, k_z \rangle. \end{equation*} Note that H\"{o}lder's inequality and Theorem \ref{FockSpacePropThm} of Section $2$ immediately implies that $B(A)$ is a bounded function on $\C$ and that $B(A)$ vanishes at infinity when $1 < p < \infty$ and $A$ is compact on $\Fpphi$.

Now suppose that $\mu$ is a complex Borel measure on $\C$ in the sense that $\mu$ can be written as $\mu = (\mu_1  - \mu_2) + i (\mu_3 - \mu_4)$ where $\mu_j, j = 1, \ldots, 4$ are positive $\sigma-$finite Borel measures on $\C$ (for example when $d\mu = f\, dv$ for $f \in L^1 _{\text{loc}}(\C)$.) Given such a complex Borel measure $\mu$ on $\C$ where $|\mu|$ is Fock-Carleson (see Section $2$ for precise definitions),  we define the Toeplitz operator $T_\mu$ with symbol $\mu$ by the equation \begin{equation*} (T_\mu f)(z) := \int_{\C} f(w) K(z, w) e^{- 2\phi(w)} \, d\mu(w) \end{equation*} where $K(z, w)$ is the reproducing kernel of $\Ftwophi$. Furthermore, if $\mu$ is given by $\mu = f \, dv$ for a measurable function $f$ on $\C$, then we write $T_f$ instead of $T_\mu$. Also note that if $|\mu|$ is Fock-Carleson, then an easy application of Fubini's theorem gives us that $B(T_\mu) = B(\mu)$ where $B(\mu)$ is the Berezin transform of $\mu$ given by \begin{equation*} (B(\mu))(z) = \int_{\C} |k_z (w)|^2 \, d\mu(w). \end{equation*}

Given any ``nice" class $X$ of complex Borel measures on $\C$ (in the previously mentioned sense), let $\Tpphi(X)$ be the $\Fpphi$ operator norm closure of the algebra generated by $\{T_\mu : \mu \in X\}$.  Furthermore, let $\SLphi$ be the class of ``sufficiently localized'' operators $A$ where $A$ is bounded on $\Fqphi$ for some $2 \leq q < \infty$ and where \begin{equation} \label{SLCond} |\langle A k_z, k_w \rangle| \lesssim \frac{1}{(1 + |z - w|)^{2n + \delta}} \end{equation} for some $\delta = \delta(A) > 0$ independent of $z, w$.  Note that $\SLphi$ includes finite sums of finite products of Toeplitz operators with Fock-Carleson measures (see Propositions \ref{AlgProp} and \ref{ToepOpInSL} in Section \ref{Preliminary results}.)  Furthermore, note that any $A \in \SLphi$ extends to a bounded operator on $\Fpphi$ for any $1 \leq p \leq \infty$ and that $\SLphi$ is also a $*$-algebra (see Section \ref{Preliminary results}.)

The following two theorems can be considered the main results of this paper.

\begin{theorem} \label{CompOpSuff}   Let $1 < p < \infty$ and let $A \in \SLphi$.  Then there exists $R = R(A) >0$ where $A$ is compact if  \begin{equation} \label{CompOpSuffCond}\limsup_{|z| \rightarrow \infty} \sup_{w \in B(z, R)} |\langle Ak_z, k_w\rangle| = 0. \end{equation}  Furthermore, if $A$ is in the $\Ftwophi$ operator norm closure of $\SLphi$ then $A$ is compact on $\Ftwophi$ when (\ref{CompOpSuffCond}) holds.\end{theorem}

\begin{theorem} \label{CompOpNec} If $1 < p < \infty$ then the space of compact operators on $\Fpphi$ coincides with $\Tpphi( C_c ^\infty(\C))$.  Furthermore, the space of compact operators on either of the spaces $\Ftwophi$ for general $\phi$ satisfying (\ref{PhiCond}) or $\Fp$ (for $1 < p < \infty$)  coincides with the operator norm closure of the set $\{T_f : f \in C_c ^\infty(\C) \}$. \end{theorem}

Note that condition (\ref{CompOpSuffCond}) is significantly weaker than the so-called ``reproducing kernel thesis condition" that often appears in the literature (see \cite{MW} for example), which says that \begin{equation*} \lim_{|z| \rightarrow \infty} \|Ak_z\|_{\Fpphi} = 0. \end{equation*}   In particular, if $1 < p < \infty$, then $f)$ and $g)$ in Theorem \ref{FockSpacePropThm} gives us that for any $R > 0$ \begin{align} \label{RepKerThesisEst} \lim_{|z| \rightarrow \infty} \sup_{w \in B(z, R)} |\langle A k_z, k_w \rangle| & \approx \lim_{|z| \rightarrow \infty} \sup_{w \in B(z, R)} | A k_z (w) | e^{-\phi(w)} \nonumber \\ & \leq \limsup_{|z| \rightarrow \infty}  \| A k_z \|_{\Finfphi} \nonumber \\ & \lesssim \limsup_{|z| \rightarrow \infty}  \| A k_z \|_{\Fpphi}. \nonumber \end{align}

However, if we assume the existence of a uniformly bounded family of operators $\{U_z\}_{z \in \C} $ on both $\Fpphi$ and $\Fqphi$ (with $q$ being the conjugate exponent of $p$) where \begin{equation} (U_z k_w) (u)  = \Theta(z, w) k_{z - w} (u)  \label{UzDef} \end{equation} with $|\Theta(\cdot, \cdot)|$ bounded above and below on $\C \times \C$, then we will give a very short and easy proof of the following result:

\begin{proposition} \label{BerProp} Assume that there exists a uniformly bounded family of operators $\{U_z\}_{z \in \C} $ on both $\Fpphi$ and $\Fqphi$ satisfying (\ref{UzDef}).  Then for any bounded $A$ on $\Fpphi$, we have that $B(A)$ vanishes at infinity if and only if $A$ satisfies (\ref{CompOpSuffCond}) for any $R > 0$.  \end{proposition}

 In the $\Fp$ setting, note that these operators are classical and in particular are the ``weighted translations" defined by \begin{equation*} (U_z h)(w) = h(z - w) k_z (w) \end{equation*} that satisfy $U_z ^* = U_z = U_z ^{-1}$. Furthermore, note that the existence of a uniformly bounded family of operators $\{U_z \}_{z \in \C}$ on both $\Fpphi$ and $\Fqphi$ satisfying (\ref{UzDef}) is often taken as an assumption when proving results about Banach or Hilbert spaces of analytic functions (see \cite{MW} for example.)  For this reason, it is rather remarkable that one can prove Theorem \ref{CompOpSuff} for $p = 2$ without assuming the existence of a uniformly bounded family of operators $\{U_z\}_{z \in \C}$ on $\Ftwophi$ satisfying (\ref{UzDef}).  Also note that Theorem \ref{CompOpNec} was proved in the $\Ft$ setting in \cite{BC}.  Despite this, it is noteworthy that both Theorems \ref{CompOpSuff} and \ref{CompOpNec} are new even in the $\Fp$ setting when $p \neq 2$.

Now if $f \in C_c ^\infty (\C)$ then it is easy to see that $T_f$ is compact on $\Fpphi$ (in fact, $T_f$ is trace class on $\Ftwophi$ if $f \in C_c ^\infty (\C)$, see the end of Section \ref{SectionProofs2} for an easy proof.) Thus, in light of Theorem \ref{CompOpNec}, Theorem \ref{BIThm} can be restated as an approximation result that says that if $A \in \Tp (L^\infty(\C))$ (which in fact as a set is equal to $\Tp (\{\mu : |\mu| \text{ is Fock-Carleson}\})$, see \cite{BI}) and $B(A)$ vanishes at infinity, then $A$ is in fact in the norm closure of the set $\{T_f : f \in C_c ^\infty(\C) \}$ when $1 < p < \infty$.

In addition to proving Theorems \ref{CompOpSuff} and \ref{CompOpNec}, we will also prove some very natural essential norm estimates for both operators in the $\Fpphi$ operator norm closure of $\SLphi$ and for general bounded operators on $\Fpphi$.  In particular, we will prove the following two theorems, the first of which is a generalization of some of the essential norm estimates in \cite{BI} and the second of which is a strong generalization of the essential norm estimates for Toeplitz operators on the unweighted Bergman space from \cite{MZ} (that were proved using vastly different techniques that the ones we use here.)  It is rather interesting to note that both of these theorems are new in the $\Fp$ setting (and in some instances are even new for $p = 2$.)

\begin{theorem}  \label{GenEssNormProp} If $1 < p < \infty$ and $A$ bounded on $\Fpphi$ then \begin{equation} \label{GenEssNormEst1} \|A\|_\Q \approx \lim_{r \rightarrow \infty} \|M_{\chi_{B(0, r) ^c}} A \|_{\Fpphi \rightarrow \Lpphi}.  \end{equation}   Moreover, if $A$ is in the $\Fpphi$ operator norm closure of $\SLphi$ then we also have \begin{equation} \label{GenEssNormEst2} \|A\|_\Q \approx \sup_{d > 0} \limsup_{|z| \rightarrow \infty} \|M_{\chi_{B(z, d)}}  A P M_{\chi_{B(z, 2d)}} \|_{\Fpphi \rightarrow \Lpphi}. \end{equation} \end{theorem}

\begin{theorem} \label{ToepEssNormThm} Let $ 0 < \delta < 1$ and let $\mu$ be a complex Borel measure where $|\mu|$ is Fock-Carleson with $\|\mu\|_* \leq 1$. If $1 < p < \infty$ then there exists $C = C_\delta$ independent of $\mu$ where
 \begin{equation*}    \|T_\mu \|_{\Q}  \lesssim         C_\delta \left(\limsup_{|z|, |w| \rightarrow \infty}   |\langle T_\mu k_z, k_w \rangle| \right)^{\delta}. \end{equation*}

 \end{theorem}

 Furthermore,  we will extend the essential norm estimates in \cite{MW} to the $p \neq 2$ case $\Fpphi$ setting, which in particular (in conjunction Proposition \ref{BerProp}) provides us with a very short proof of Theorem \ref{BIThm} for $p \neq 2$  when compared to the proof of Theorem \ref{BIThm} from \cite{BI} (note that a similar simplification when $p = 2$ was also provided in \cite{MW}.)

 We will now briefly outline the structure of this paper. The next section will discuss some preliminary results that will be used throughout the rest of the paper (including a brief discussion of Fock-Carleson measures and the short proof of Proposition \ref{BerProp}.)  In Section \ref{SectionProofs1}, we will prove Theorem \ref{CompOpSuff}  when $p = 2$.  Although the proof of this result uses the frame theoretic ideas from \cite{XZ}, the details of the arguments in Section \ref{SectionProofs2} are considerably simpler and more transparent than the details in \cite{XZ}.  Section \ref{SectionProofs2} will contain the proof of Theorem \ref{CompOpNec}, and in Section \ref{EssNormSec} we will prove Theorem \ref{CompOpSuff} when $p \neq 2$ and also prove Theorems \ref{GenEssNormProp} and \ref{ToepEssNormThm} by extending the ideas and essential norm estimates from \cite{MW} to the $\Fpphi$ setting.  Finally Section \ref{OpenProblems} will discuss interesting open questions related to the results of this paper.

 Note that we will write $A \lesssim B$ for two quantities $A$ and $B$ if there exists an unimportant constant $C$ such that $A \leq C B$.  Furthermore, $B \lesssim A$ is defined similarly and we will write $A \approx B$ if $A \lesssim B$ and $B \lesssim A$.

Finally in this introduction we will briefly discuss a concrete and interesting (from the point of view of holomorphic function spaces) example of a generalized Fock space.  In particular, we will now show that the Fock-Sobolev spaces introduced recently in \cite{CZ} are in fact generalized Fock spaces. Given any $m \in \N$, let $F_\alpha ^{p, m}$ denote the Fock-Sobolev space of entire functions with the norm \begin{equation*} \|f\|_{F_\alpha ^{p, m}} :=\sum_{|\beta| \leq m} \|\partial ^\beta f \|_{\Fp} \end{equation*} where the sum is over all multiindicies $\beta$ with $|\beta| \leq m$.   It was then proved in \cite{CZ} that $f \in F_\alpha ^{p, m}$ if and only if $z\mapsto |z|^m f(z) \in \Lp$ where $\Lp := \Lpphi$ for $\phi(z) = \frac{\alpha}{2} |z|^2$, and furthermore the canonical norms induced by these two conditions are equivalent (note that this was only proved for $\alpha = 1$ but the extension to general $\alpha > 0 $ is trivial.)

By a standard closed graph theorem argument, we have that $f \in F_\alpha ^{p, m}$ if and only if $z \mapsto (A + |z|^2 )^\frac{m}{2} f(z) $ is in $\Lp$ for any $A > 0$, and furthermore the canonical norm induced by this condition (for fixed $A > 0$) is equivalent to the $F_\alpha ^{p, m}$ norm.   Thus, if \begin{equation*} \phi(z) := \frac{\alpha}{2} |z|^2 - \frac{m}{2} \ln (A + |z|^2)  \end{equation*} then we have $F_\alpha ^{p, m} = \Fpphi$ and \begin{equation*}  d d^c \phi (z) = \sum_{j, k = 1}^n  \left(\frac{\alpha}{4} \delta_{kj}-  \frac{m ((A + |z|^2) \delta_{kj} - z_j \overline{z_k})}{4(A + |z|^2)^2} \right) dz_k \wedge d\overline{z_j} \end{equation*}  which by an application of the Cauchy-Schwartz inequality gives us that \begin{equation*}   \left(\frac{\alpha}{4} -  \frac{m}{4(A + |z|^2)}\right) \omega_0   \leq   d d^c \phi  \leq \left(\frac{\alpha}{4} -  \frac{m A}{4(A + |z|^2)^2}\right) \omega_0. \end{equation*}  Thus, we have that $\phi$ satisfies condition (\ref{PhiCond}) if $A > 2m/\alpha$.  Because of this, the reader should keep in mind that all of the results proved in this paper also apply to Fock-Sobolev spaces (which by themselves for Fock-Sobolev spaces are interesting in their own right.) \\

\textit{Remark: } Well after this paper was written, the author in collaboration with B. Wick and M. Mitkovski was able to prove that if $1 < p < \infty$ and $A$ is in the $\Fpphi$ operator norm closure of $\SLphi$, then there exists $R = R(A) > 0$ where $A$ is compact if (\ref{CompOpSuffCond}) is true.  In fact, one can even replace the conditions defining $\SLphi$ by similar but weaker integral conditions (see \cite{IMW} for details.)

\section{Preliminary results} \label{Preliminary results}  In this section, we will state and prove some preliminary results that will be used in the rest of the paper.  First, we will mention some important properties of $\Fpphi$ from \cite{SV} that should be familiar to the reader who has experience with the classical Fock spaces $\Fp$.

\begin{theorem}  \label{FockSpacePropThm} The Fock spaces $\Fpphi$ satisfy the following properties:
\begin{list}{}{}
\item $a) \   \text{There exists } \epsilon, C > 0$ independent of $z, w \in \C$ such that
\noindent  \begin{equation*} e^{- \phi(z) } |K(z, w)| e^{-\phi(w)} \leq C e^{-\epsilon |z - w|}. \end{equation*}  \\
\item $b) \  \text{If } 1 \leq p \leq \infty$ then $k_z \rightarrow 0$ weakly in $\Fpphi$ as $|z| \rightarrow \infty$.  \\
\item $c) \  \text{If } 1 \leq p < \infty$ then $(\Fpphi)^* = \Fqphi$ for $1/p + 1/q = 1$ under the usual pairing \begin{equation*} \Psi_g (f) := \int_{\C} f(z) \overline{g(z)}  e^{-2\phi(z)} \, dv(z). \end{equation*}  \\
\item $d) \  \text{The orthogonal projection } P: \Ltwophi \rightarrow \Ftwophi$ extends to a bounded operator from $\Lpphi$ to $\Fpphi$ when $1 \leq p \leq \infty$. \\
 \item $e)  \  P$ restricted to $\Fpphi$ is the identity operator when $1 \leq p \leq \infty$. \\
 \item $f) \ e^{\phi(z)} \approx \sqrt{K(z, z)}$ for any $z \in \C$. \\
\item $g) \ $If $ \ 0 < p < \infty$ and $r > 0$ then there exists $C_r > 0$ where \begin{equation*} (|f|^p e^{-p \phi})(z)   \lesssim C_r \int_{B(z, r)} |f(w)|^p e^{- p \phi(w)} \, dv(w) \end{equation*}  and \begin{equation*} | \nabla (|f|^p e^{-p \phi})| (z)   \lesssim C_r \int_{B(z, r)} |f(w)|^p e^{- p \phi(w)} \, dv(w) \end{equation*}  for any $f \in \Fpphi$ and $z \in \C$. \\
 \end{list}
\end{theorem}

Note that property $a)$ immediately implies that $\{k_z : z \in \C\}$ is bounded in $\Fpphi$ when $0 < p \leq \infty$.  Furthermore, note that property $a)$ for the classical Fock space $\Ft$ is in fact true for any $\epsilon > 0$.  In particular, since \begin{equation*} K^\alpha  (z, w) = e^{\alpha (z \cdot \overline{w})} \end{equation*}where $K^\alpha  (z, w) $ is the reproducing kernel of $\Ft$,  we have that \begin{align*} e^{- \frac{\alpha}{2} |z|^2 } |K^\alpha  (z, w)| e^{- \frac{\alpha}{2} |w|^2 }  & = e^{- \frac{\alpha}{2} |z|^2 } |e^{ \alpha (z \cdot \overline{w})}|  e^{- \frac{\alpha}{2} |w|^2 } \\ & = e^{- \frac{\alpha}{2}  |z - w|^2}. \end{align*}  In general however, one can not expect to have such a fast off diagonal decay when dealing with generalized Fock spaces (though fortunately, as was noticed in \cite{SV}, quadratic exponential off diagonal decay as above is usually not needed.)

Now  if $\nu$ is a nonnegative Borel measure on $\C$, then we say $\nu$ is a Fock-Carleson measure for $\Fpphi$ if the embedding operator $\iota : \Fpphi \rightarrow \Lpphi(\nu)$  is bounded. We will often use the following useful characterization of Fock-Carleson measures on $\C$ (see \cite{SV} for a proof.)

\begin{theorem} \label{FockCarThm} If $1 \leq p < \infty$ and $\nu$ is a nonnegative Borel measure, then the following are equivalent:

\begin{list}{}{}
          \item $a) \ \nu$ is a Fock-Carleson measure for $\Fpphi$,
          \item $b) \ \sup_{z \in \C} \nu(B(z, 1))  < \infty$,
          \item $c) \ T_\nu$ is bounded on $\Fpphi$.
          \end{list}

Furthermore, the canonical norms defined by any of these three conditions are equivalent.
\end{theorem}

\noindent Since $\nu$ being Fock-Carleson for $\Fpphi$ is independent of $p$ when $1 \leq p < \infty$, we will say $\nu$ is a Fock-Carleson measure if $\nu$ satisfies any of the equivalent conditions in Theorem \ref{FockCarThm}. Furthermore, if $\mu$ is a complex Borel measure on $\C$, then we will denote by $\|\mu\|_*$ any of the canonical norms applied to the variation measure $|\mu|$ defined by the conditions in Theorem \ref{FockCarThm} . We will also let $\|f\|_*$ denote $\||f| \, dv\|_*$ when $f$ is a measurable function on $\C$.

We will now show that the spaces $\Fpphi$ behave in the same way that the spaces $\Fp$ do under complex interpolation (see \cite{Zhu}.)

\begin{theorem} \label{CompInterpTHM}  If $1 \leq p_0 \leq p_1 \leq \infty$ and $0 \leq \theta \leq 1$ where \begin{equation*} \frac{1}{p} = \frac{1 - \theta}{p_0} + \frac{\theta}{p_1} \end{equation*}  then \begin{equation*} {[F_{\phi} ^{p_0} , F_{\phi} ^{p_1}]}_{\theta} = \Fpphi \end{equation*} with equivalent norms. \end{theorem} \noindent

\begin{proof}
First note that the classical Stein-Weiss interpolation theorem gives us that \begin{equation} \label{CompInterpPROP} {[L_{\phi} ^{p_0} , L_{\phi} ^{p_1}]}_{\theta} = \Lpphi \end{equation}  with equal norms. Now by the definition of ${[L_{\phi} ^{p_0} , L_{\phi} ^{p_1}]}_{\theta}$ and (\ref{CompInterpPROP}), we have that  $ {[F_{\phi} ^{p_0} , F_{\phi} ^{p_1}]}_{\theta} \subseteq \Fpphi$.

On the other hand, if $f \in \Fpphi \subseteq \Lpphi$, then again by (\ref{CompInterpPROP}) there exists a positive constant $C$ and an analytic function $w \mapsto F(\cdot, w)$ from  $  \{w \in \mathbb{C} : 0 \leq \text{Re } w \leq 1\} $ to $  L_{\phi} ^{p_0} +  L_{\phi} ^{p_1}$ where \\

\begin{list}{}{}
\item $a) \  F(z, \theta) = f(z)$ for all $z \in \C$,
\item $b) \ \|F(\cdot, w) \|_{L_{\phi} ^{p_0}} \leq C$ for all $\text{Re } (w) = 0$,
\item $c) \ \|F(\cdot, w) \|_{L_{\phi} ^{p_1}} \leq C$ for all $\text{Re } (w) = 1$.
\end{list}

\noindent Now let $G(z, w) = (P (F(\cdot, w)))(z)$. Then by $a)$ and $d)$ in Theorem \ref{FockSpacePropThm} and Morera's theorem, we have that $w \mapsto G(\cdot, w)$ is an analytic function from  $ \{w \in \mathbb{C} : 0 \leq \text{Re } w \leq 1\} $ to $  F_{\phi} ^{p_0} +  F_{\phi} ^{p_1}$ and $G$ satisfies \\

\begin{list}{}{}
\item $a) \  G(z, \theta) = (Pf)(z) = f(z)$ for all $z \in \C$,
\item $b) \ \|G(\cdot, w) \|_{L_{\phi} ^{p_0}} \leq C'$ for all $\text{Re } (w) = 0$,
\item $c) \ \|G(\cdot, w) \|_{L_{\phi} ^{p_1}} \leq C'$ for all $\text{Re } (w) = 1$.
\end{list}

\noindent for some positive constant $C'$, which implies that $f \in {[F_{\phi} ^{p_0} , F_{\phi} ^{p_1}]}_{\theta}$, or ${[F_{\phi} ^{p_0} , F_{\phi} ^{p_1}]}_{\theta} = \Fpphi$.

To show the equivalence of norms, let $f \in \Fpphi$.  Then by definition (\ref{CompInterpPROP}) we have that

\begin{equation*} \|f\|_{\Fpphi} = \|f\|_{{[L_{\phi} ^{p_0} , L_{\phi} ^{p_1}]}_{\theta}} \leq \|f\|_{{[F_{\phi} ^{p_0} , F_{\phi} ^{p_1}]}_{\theta}}. \end{equation*} An application of the open mapping theorem immediately completes the proof.
\end{proof}

We next prove some simple results regarding $\SLphi$.  Note that the proof of the second one is similar to the proof of Proposition $3.2$ in \cite{XZ} though we include the details for the sake of the reader.

\begin{proposition} \label{SLphiBddProp} Any operator $A \in \SLphi$ extends bounded on $\Fpphi$ for any $1 \leq p < \infty$.  \end{proposition}

\begin{proof} Assume that $A$ is bounded on $\Fqphi$ for some $2 \leq q < \infty$ and that $A$ satisfies $(\ref{SLCond})$.
Let $g \in F_\phi ^1 \subseteq \Fqphi$ and note that an application of $(\ref{SLCond})$ and $f)$ in Theorem \ref{FockSpacePropThm} gives us that
\begin{align*} |(Ag)(w)| e^{-\phi(w)} & \approx |\langle g, A^* k_w \rangle| \\ & \leq
\incn |g(u)| |(A^* k_w)(u) | e^{-2\phi(u)} \, dv(u) \\ & \approx \incn \left(|g(u)| e^{-\phi(u)}\right) |\langle Ak_u, k_w \rangle | \, dv(u) \\ & \lesssim \incn \left(|g(u)| e^{-\phi(u)}\right) \frac{1}{(1 + |u - w|)^{2n + \delta}} \, dv(u). \end{align*}  An easy application of Fubini's theorem then gives us that $A$ extends to a bounded operator on $F_\phi ^1$, and furthermore since $2 \leq q < \infty$, Theorem \ref{CompInterpTHM} gives us that $A$ extends to a bounded operator on $\Fpphi$ for all $1 \leq p \leq 2$.

Finally, this means that $A^*$ is bounded on $\Fpphi$ for all $2 \leq p < \infty$. In particular, we have that $A^* \in \SLphi$ and so $A^*$ is bounded on $\Fpphi$ for all $1 \leq p \leq 2$, which implies that $A$ is bounded on $\Fpphi$ for all $1 \leq p < \infty.$ \end{proof}

\begin{proposition} \label{AlgProp} $\SLphi$ is a $*-algebra$.  \end{proposition}

\begin{proof} As was remarked in the proof of Proposition \ref{SLphiBddProp}, $A \in \SLphi \Longrightarrow A^* \in \SLphi$.  Thus, we only need to show that $\SLphi$ is an algebra.

Let $A_1, A_2 \in \SLphi$ and pick $\delta_i> 0 $ where \begin{equation*} |\langle
 A_i k_z, k_w \rangle| \lesssim \frac{1}{(1 + |z - w|)^{2n + \delta_i}} \end{equation*} for $i = 1, 2$.   Then by Theorem \ref{FockSpacePropThm} we have \begin{align*}
 |\langle A_1 A_2 k_z, k_w \rangle| & \leq  \int_{\C} |(A_2 k_z) (u)| |(A_1 ^* k_w ) (u)| e^{-2\phi(u)} \, du \\ & \lesssim
 \int_{\C} |\langle A_2 k_z, k_u \rangle| |\langle A_1 k_u, k_w \rangle| \, du \\ & \lesssim \int_{\C} \frac{1}{(1 + |z - u|)^{2n + \delta_2} (1 + |u - w|)^{2n + \delta_1}}  \, du \end{align*} The proof now easily follows with $\delta(A_1 A_2) = \min\{\delta_1, \delta_2\}$ (see  p. 10 in \cite{XZ} for more details.)    \end{proof}

\begin{proposition} \label{ToepOpInSL} If $\mu$ is a complex Borel measure on $\C$ such that $|\mu|$ is Fock-Carleson then $T_\mu \in \SLphi$ for any $1 < p < \infty$.  \end{proposition}

\begin{proof}  First note that Theorem \ref{FockCarThm} gives us that $T_\mu$ is bounded on $\Fpphi$ for all $1 \leq p < \infty$.  Now note that Fubini's theorem and Theorems \ref{FockSpacePropThm} and \ref{FockCarThm} tell us that \begin{equation*} \langle T_\mu k_z, k_w \rangle = \int_{\C} k_z(u) \overline{k_w (u)}  e^{-2\phi(u)} \, d\mu(u). \end{equation*}  Furthermore, another easy application of Theorems \ref{FockSpacePropThm} and \ref{FockCarThm} and the fact that $k_z (\cdot) k_w (\cdot) \in F^1 _{2\phi}$  for each $z, w \in \C$ gives us that \begin{align*}
 |\langle T_\mu k_z, k_w \rangle | & \leq  \int_{\C} |k_z(u)| |k_w (u)|  e^{-2\phi(u)} \, d|\mu| (u)  \\ &  \lesssim \|\mu\|_* \int_{\C} |k_z(u)| |k_w (u)|  e^{-2\phi(u)} \, dv(u) \\ & \lesssim \|\mu\|_* \int_{\C} e^{-\epsilon (|z - u| + |u - w|)} \, dv(u) \\ & \lesssim \|\mu\|_* e^{-\frac{ \epsilon}{2} |z - w|}. \end{align*}
 \end{proof}

Finally in this section we prove Proposition \ref{BerProp}.

 \noindent \textit{Proof of Proposition \ref{BerProp}}. If $R > 0$ then obviously we have \begin{equation*}
 |(B(A))(z)| = |\langle Ak_z, k_z\rangle| \leq \sup_{w \in B(z, R)} |\langle Ak_z, k_w\rangle| \end{equation*} so that $B(A)$ vanishes at infinity if (\ref{CompOpSuffCond}) is true.

Now assume the existence of a uniformly bounded family of operators on both $\Fpphi$ and $\Fqphi$ satisfying (\ref{UzDef}). Furthermore, assume that $B(A)$ vanishes at infinity but that \begin{equation*} \limsup_{|w| \rightarrow \infty} \sup_{w \in B(z, R)} |\langle Ak_z, k_w\rangle| \neq 0 \end{equation*} for some fixed $R > 0$.  Thus, there exists sequences $\{z_m\}, \{w_m\} $  where $\lim_{m \rightarrow \infty} |z_m| = +\infty$ and $|w_m| \leq R$ for any $m \in \N$, and where \begin{equation} \label{BerPropAssump} \limsup_{m \rightarrow \infty} |\langle Ak_{z_m}, k_{z_m - w_m}\rangle| > \epsilon\end{equation} for some $\epsilon > 0$.  Furthermore, passing to a subsequence if necessary, we may assume that $\lim_{m \rightarrow \infty} w_m = w$. Note that an easy application of Theorem \ref{FockSpacePropThm} and the Lebesgue dominated convergence theorem gives us that $\lim_{m \rightarrow \infty} k_{w_m} = k_w $ in where the convergence is in the $\Fpphi$ norm.

Let $\mathcal{B}(\Fpphi)$ be the space of bounded operators on $\Fpphi$.  Now since $(\Fpphi)^* = \Fqphi$, an argument that is almost identical to the proof of the Banach-Alaoglu theorem tells us that the unit ball of $\mathcal{B}(\Fpphi)$ is $\WOT$ compact.  Then passing to another subsequence if necessary, we can assume \begin{equation*} \widehat{A} = \WOT -  \lim_{m \rightarrow \infty} U_{z_m} ^* A U_{z_m}. \end{equation*} Thus, we have that \begin{align*} \limsup_{m \rightarrow \infty} |\langle Ak_{z_m}, k_{z_m - w_m}\rangle| & \approx
\limsup_{m \rightarrow \infty} |\langle U_{z_m} ^*A U_{z_m} k_{0},  k_{ w_m}\rangle| \\ & = \limsup_{m \rightarrow \infty} |\langle U_{z_m} ^*A U_{z_m} k_{0},  k_{ w}\rangle| \\ & = |\langle\widehat{A} k_0, k_w \rangle| .\end{align*}    However, for any $z \in \C$ \begin{equation*} |\langle \widehat{A} k_z, k_z \rangle| = \lim_{m \rightarrow \infty}  |\langle  U_{z_m} ^* A U_{z_m} k_z, k_z \rangle| \approx \lim_{m \rightarrow \infty} |\langle A k_{z_m - z}, k_{z_m - z} \rangle| = 0 \end{equation*} since by assumption $B(A)$ vanishes at infinity. Thus, since the Berezin transform is injective (see the end of Section \ref{SectionProofs2}), we get that $\widehat{A} = 0$, which contradicts (\ref{BerPropAssump}) and completes the proof.

\section{Proof of Theorem \ref{CompOpSuff} for $p = 2$} \label{SectionProofs1}
In this section we will prove Theorem \ref{CompOpSuff} when $p = 2$.  Now if $f \in \Ftwophi$, then note that Fubini's theorem and Theorem \ref{FockSpacePropThm} gives us that \begin{align*} f(w) & = \int_{\C} f(z) K(w, z) \, e^{-2\phi(z)} \, dv(z) \\ & = \int_{\C} f(z) \langle K(\cdot, z) , K(\cdot, w) \rangle  \, e^{-2\phi(z)} \, dv(z) \\ & = \incn ((\Tk_z \otimes \Tk_z) f)(w)  \, dv(z)  \end{align*} where \begin{equation*} \Tk_z = e^{-\phi(z)} K(\cdot, z). \end{equation*}

In other words, we have that \begin{equation} \label{ResOfId} \text{Id}_{\Ftwophi \rightarrow \Ftwophi} = \incn \Tk _z \otimes \Tk _z \, dv(z) \end{equation} where the integral is interpreted as a standard B\^{o}chner integral, which roughly states that we can treat $\{\Tk_z\}_{z \in \C}$ as a sort of continuously indexed frame.  Furthermore, note that $\int_K \Tk _z \otimes \Tk _z \, dv(z)$ is compact (Hilbert-Schmidt in fact) on $\Ftwophi$ for any compact $K \subseteq \C$.

We will now very briefly sketch the main idea of the proof of Theorem \ref{CompOpSuff} when $p = 2$. First, with the help of some simple ideas from classical frame theory, we will rewrite (\ref{ResOfId}) in a kind of discretized way that is more convenient for us.  We will then combine this with the fact that operators in $\SLphi$ are ``almost diagonal" with respect to $\{\Tk_z \}_{z \in \C}$ to prove that $\|A\|_\Q$ can be dominated by the norm of a certain block diagonal matrix involving the family $\{A\Tk_z\}_{z \in \C}$ if $A$ is in the $\Ftwophi$ operator norm closure of $\SLphi$.  Finally, we will complete the proof of Theorem \ref{CompOpSuff} when $p = 2$ by showing that condition \ref{CompOpSuffCond} easily implies that the norm of these blocks approaches zero as one goes farther down the diagonal.  As was mentioned in the introduction, this same idea was used to prove Theorem \ref{CompOpSuff}  in the $\Ft$ setting.  Despite this, it is again worth noting that the details of the arguments in this section are considerably simpler and more transparent than the details in \cite{XZ}.

Now treat $\Z$ as a lattice in $\C$ in the canonical way and let $\{e_u\}_{u \in \Z}$ be any fixed orthonormal basis for $\Ftwophi$.    Note that by f) in Lemma \ref{FockSpacePropThm} we have that $|\Tk_z(w)| \approx |k_z(w)|$ for any $z, w \in \C$.

The proof of Theorem \ref{CompOpSuff} when $p = 2$ will require the following three lemmas, the first of which is well known (though we include the proof for the sake of completion), and the third of which contains the essential ideas of the proof. Note that in this section, all norms will either be the $\Ftwophi$ norm, or the operator norm on $\Ftwophi$.

\begin{lemma} \label{XiaLem1} If $F_z : = \sum_{u \in \Z} \Tk_{u + z} \otimes e_u$ is the translated ``pre-frame operator" asociated to $\{\Tk_{u + z}\}_{u \in \C}$ for $z \in \C$, then $\sup_{z \in \C} \|F_z \| < \infty.$
\end{lemma}

\begin{proof} An easy computation gives us that \begin{equation*} F_z F_z^*  = \sum_{u \in \Z} \Tk_{u + z} \otimes \Tk_{u + z}. \end{equation*} Thus, f) and g) in Lemma \ref{FockSpacePropThm} gives us that \begin{align*} \langle F_z F_z^* f, f \rangle& = \sum_{u \in \Z} |\langle f, \Tk_{u + z} \rangle|^2 \\ & = \sum_{u \in \Z} |f(u + z)|^2 e^{-2\phi(u + z)} \\ & \lesssim \sum_{u \in \Z} \int_{B(u + z, \frac{1}{2})} |f(w)|^2 e^{-2\phi(w)} \, dv(w) \\ & \leq \|f\|^2 \end{align*} if $f \in \Ftwophi$.   \end{proof}

\begin{lemma} \label{XiaLem2} Suppose that $B \in \SLphi$ and let $\epsilon > 0$.  Then there exists $R = R(B, \epsilon)$ where if $\Omega \subset \Z \times \Z$ satisfies $|u - v| \geq R$ for any $(u, v) \in \Omega$ and if $\eta, \xi \in S := [0, 1)^{2n} \subset \C$,  then \begin{equation*}
\left\| \sum_{(u, v) \in \Omega} \langle B\Tk_{v + \eta}, \Tk_{u + \xi} \rangle \ e_u \otimes e_v \right\| \leq \epsilon \end{equation*} \end{lemma}

\begin{proof} Without loss of generality assume that $R \geq 1$ so that $(u, v) \in \Omega$ implies that $|u - v| \geq 1$.  Since $|u - v| \geq R$ for any $(u, v) \in \Omega$,  we immediately obtain that \begin{equation*} \left|\langle B\Tk_{v + \eta}, \Tk_{u + \xi} \rangle \right| \lesssim \frac{1}{(1 + R^{\frac{\delta}{2}}) |u - v|^{2n + \frac{\delta}{2}}} \end{equation*} for any $\eta, \xi \in S$.

Now let $p_i : \Omega \rightarrow \Z$ for $i = 1, 2$ be the projection onto the i${}^\text{th}$ factor.  Furthermore, for each $u \in p_1 (\Omega)$ and each integer $\ell \geq 0$, let \begin{equation*} G_\ell ^u := \{v : (u, v) \in \Omega  \text{ and }  2^\ell \leq |u - v| < 2^{\ell + 1} \}. \end{equation*}

By an elementary volume count, we have that \begin{equation*} \text{card } G_\ell ^u \lesssim 2^{2n\ell}. \end{equation*}
Thus, for any $u \in p_1 (\Omega)$ we have \begin{align*} \sum_{v :  (u, v) \in \Omega} \left|\langle B\Tk_{v + \eta}, \Tk_{u + \xi} \rangle \right| & \lesssim \frac{1}{(1 + R)^{\frac{\delta}{2}}}  \sum_{v :  (u, v) \in \Omega} \frac{1}{(1 + |u - v|)^{2n + \frac{\delta}{2}}} \\ & = \frac{1}{(1 + R)^{\frac{\delta}{2}}} \sum_{\ell = 0}^\infty \sum_{v \in G_\ell ^u} \frac{1}{(1 + |u - v|)^{2n + \frac{\delta}{2}}} \\ & \lesssim \frac{1}{(1 + R)^{\frac{\delta}{2}}} \sum_{\ell = 0}^\infty \frac{2^{2n\ell}}{2^{(2n + \frac{\delta}{2})\ell}}  \\ & \lesssim \frac{1}{(1 + R)^{\frac{\delta}{2}}}. \end{align*} Similarly, since $B^* \in \SLphi$, we have for each $v \in p_2(\Omega)$ that \begin{equation*} \sum_{u : (u, v) \in \Omega} \left|\langle B^*\Tk_{v + \eta}, \Tk_{u + \xi} \rangle \right| \lesssim \frac{1}{(1 + R)^\frac{\delta}{2}}. \end{equation*}

\noindent Therefore, an easy application of the Schur test now completes the proof.

\end{proof}

For the next lemma, it will be convenient to use the standard ``sup-norm" $|\cdot|_{\infty} $ on $\C$ defined for $z = (z_1, \ldots, z_n)$ by \begin{equation*} \INF{z} := \max\{|z_1|, \ldots, |z_n|\} \end{equation*} and we will let $B_{\infty}(z, R)$ denote the open ball in $\C$ with center $z \in \C$ and radius $R > 0$ under this norm.  Furthermore, for any $R > 0$ let \begin{equation*} R \Z := \{Ru : u \in \Z\} \end{equation*} and     let \begin{equation*} \Z_R := \{u \in \Z : |u|_\infty < R \} \end{equation*} Also, for $z \in \C$ and $R \in \N$ let $F_{z;R}$ denote the translated and truncated ``pre-frame operator" defined by \begin{equation*} F_{z;R} := \sum_{u \in \Z_R } \Tk_{u + z} \otimes e_u. \end{equation*}  Note that if $A$ is bounded on $\Ftwophi$ and $a, b \in \C$, then by definition we have that \begin{equation*} F_{a;R} ^* A F_{ b ; R} = \sum_{x, y \in \Z_R} \langle A \Tk_{y + b} , \Tk_{x + a} \rangle e_x \otimes e_y. \end{equation*}

\begin{lemma} \label{XiaLem3} For any $A$ in the $\Ftwophi$ operator norm closure of $\SLphi$, there exists some $R \in \N$ depending on $A$ where the following holds for any $N \in \N$: there exists $a, b \in \C$ with \begin{equation*} \INF{a} \geq N - 1 \  \ \text{ and } \ \ \INF{b} \leq 2  \end{equation*} such that \begin{equation*} \|A\|_{\Q} \lesssim \|F_{a;R}^* A F_{a + b ; R}\|. \end{equation*}
\end{lemma}

\begin{proof}  Obviously we may assume that $\|A\|_\Q > 0$ for otherwise there is nothing to prove.  We will now in fact find $a$ and $b$ as above where \begin{equation*} \|A\|_{\Q} \leq \frac{1}{4^{n + 3} C^2}   \|F_{a;R}^* A F_{a + b ; R}\|\end{equation*} and where \begin{equation}\label{Cdef} C := \sup_{z \in \C} \|F_z\|   \end{equation} (which is finite by Lemma \ref{XiaLem1}.)
To that end,  pick $B \in \SLphi$ where \begin{equation} \label{XiaLem3Est1}
\|A - B\| < \frac{1}{4^{n+3} C^4}  \|A\|_\Q. \end{equation}    Since $B \in \SLphi$, Lemma \ref{XiaLem2} tells us that there exists $R > 0$ where \begin{equation} \label{XiaLem3Est2}
\left\| \sum_{(u, v) \in \Omega} \langle B\Tk_{v + \eta}, \Tk_{u + \xi} \rangle \ e_u \otimes e_v \right\| \leq \frac{1}{16C^2} \|A\|_\Q \end{equation} whenever $\eta, \xi \in S$ and $\Omega \subset \Z \times \Z$ satisfies $|u - v|_\infty \geq R$ for any $(u, v) \in \Omega$.  We will in fact show that this $R$ has the desired property.

Clearly without loss of generality we may assume that $N > R$.  Now define the compact operator $K$ on $\Ftwophi$ by  \begin{equation*} K := \sum_{\substack{u \in \Z \\ \INF{u} < N + R}}\int_{S + u} \Tk_z \otimes \Tk_z \, dv(z) = \int_S \left(\sum_{\substack{u \in \Z \\ \INF{u} < N + R}} \Tk_{u + z} \otimes \Tk_{u + z} \, \right) dv(z)  \end{equation*} where as before $S = [0, 1)^{2n} \subset \C$.  Note that (\ref{ResOfId}) then tells us that we can write $\text{Id} - K$ as  \begin{equation*} \text{Id} - K = \int_{S} \left(\sum_{\substack{u \in \Z \\ \INF{u} \geq N + R}} \Tk_{u + z} \otimes \Tk_{u + z} \, \right) dv(z). \end{equation*}

Thus, if we define \begin{equation*} G_z := \sum_{\substack{u \in \Z \\ \INF{u} \geq N + R}} \Tk_{u + z} \otimes e_u \end{equation*} then \begin{equation*} \text{Id} - K = \int_{S} G_z G_z ^* \, dv(z). \end{equation*}
Since (\ref{ResOfId}) again gives us that \begin{equation*} \text{Id} = \int_{S} \left( \sum_{u \in \Z} \Tk_{u + z} \otimes \Tk_{u + z} \right) \, dv(z) =\int_{S} F_z F_z ^* \, dv(z), \end{equation*} we can rewrite $(\text{Id} - K)A$ as \begin{equation} \label{XiaLem3Est3} (\text{Id} - K)A = (\text{Id} - K)A \text{Id} = \int_{S} \int_{S} G_z G_z ^* A F_w F_w ^* \, dv(z) \, dv(w). \end{equation}

Now since $\|A\|_\Q = \|(\text{Id} - K)A\|_\Q$, an elementary approximation argument involving B\^{o}chner integrability in conjunction with (\ref{XiaLem3Est3}) gives us a pair $z_0, w_0 \in S$ where \begin{equation*} \label{XiaLem3Est4}  \|G_{z_0} G_{z_0} ^* A F_{w_0} F_{w_0} ^* \|_{\Q}  \geq \frac{1}{2} \|A\|_{\Q}. \end{equation*}  Furthermore, it is trivial that $G_{z_0} G_{z_0} ^* \leq F_{z_0} F_{z_0} ^*$ so by (\ref{XiaLem3Est1}) we have \begin{equation*} \|G_{z_0} ^* B F_{w_0}\|_{\Q} + \frac{1}{64C^2} \|A\|_{\Q} \geq \|G_{z_0} ^* A F_{w_0}\|_{Q} \geq \frac{1}{2 C^2} \|A\|_\Q \end{equation*} (where $C$ is from (\ref{Cdef}))  so that that \begin{equation*} \|G_{z_0} ^* B F_{w_0} \|_\Q \geq \frac{1}{4C^2} \|A\|_\Q. \end{equation*}

Now since \begin{equation*} G_{z_0} ^* B F_{w_0} = \sum_{\substack{ \eta \in \Z \\ |\eta|_{\infty} \geq N + R}} \sum_{u \in \Z} \langle B \Tk_{u + w_0}, \Tk_{\eta + z_0} \rangle e_\eta \otimes e_u,  \end{equation*} we can write $G_{z_0} ^*  B F_{w_0} = D + E$ where the ``diagonal" part $D$ is given by \begin{equation*} \sum_{\substack{ \eta \in \Z \\ |\eta|_{\infty} \geq N + R}} \left(\sum_{\substack{u \in \Z \\    |u - \eta|_\infty < R}} \langle B \Tk_{u + w_0}, \Tk_{\eta + z_0} \rangle e_\eta \otimes e_u\right)  \end{equation*} and the ``off-diagonal" part $E$ is given by \begin{equation*} \sum_{\substack{ \eta \in \Z \\ |\eta|_{\infty} \geq N + R}} \left(\sum_{\substack{u \in \Z \\    |u - \eta|_\infty \geq R}} \langle B \Tk_{u + w_0}, \Tk_{\eta + z_0} \rangle e_\eta \otimes e_u\right) \end{equation*}

\noindent Note that (\ref{XiaLem3Est2}) gives us that \begin{equation*} \|E\| \leq \frac{1}{8C^2} \|A\|_{\Q} \end{equation*}

\noindent so that \begin{equation*}  \|D\|_{\Q} \geq \frac{1}{8C^2} \|A\|_\Q \end{equation*}

 Now by elementary arguments, we have that \begin{equation*} \{(\eta, u)  \in \Z \times \Z : \INF{\eta} \geq N + R \text{ and } \INF{\eta - u} < R \} = A_1 \backslash A_2 \end{equation*} where \begin{align*} A_1 := \{(x + u', &  y + u') \in \Z \times \Z : u' \in R \Z  \\ & \text{ with } \INF{u'} \geq N  \text{ and } (x, y) \in \Z_R \times \Z_R \}\end{align*} and \begin{align*} A_2 := \{(x + u', & y + u') \in \Z \times \Z \\ &  \text{ with } (x,y) \in \Z_R \times \Z_R \text{ and } \INF{x + u'} < N + R \text{ or } \INF{x - y} \geq R\}. \end{align*}

Thus, we can write $D := D_1 - D_2$ where \begin{equation*} D_1 = \sum_{(\eta, u) \in A_1} \langle B \Tk_{u + w_0}, \Tk_{\eta + z_0} \rangle e_\eta \otimes e_u \end{equation*} and \begin{equation*} D_2 = \sum_{(\eta, u) \in A_1 \cap A_2} \langle B \Tk_{u + w_0}, \Tk_{\eta + z_0} \rangle e_\eta \otimes e_u \end{equation*}

Moreover, another application of (\ref{XiaLem3Est2}) gives us that \begin{equation*} \|D_2\|_\Q \leq \frac{1}{16C^2} \|A\|_\Q \end{equation*} so that \begin{equation*}   \|D_1\| \geq \frac{1}{16C^2} \|A\|_\Q \end{equation*}
If  \begin{equation*}  E_{u, z} := \sum_{x \in \Z_R} \Tk_{x + u + z} \otimes e_{x + u}   \end{equation*} for some given $u \in \Z$ and $z \in \C$ then note that we can write \begin{equation*} D_1 = \sum_{\substack{u \in R\Z \\ \INF{u} \geq N  }} E_{u, z_0} ^* B E_{u, w_0}. \end{equation*}

Now let $\mathbb{Z}_1 $ and $\mathbb{Z}_2$ denote the odd and even integers, respectively, and for $\ell \in \{1, 2\}^{2n}$ let $\Z _\ell := \mathbb{Z}_{\ell_1} \times \cdots \times \mathbb{Z}_{\ell_{2n}}$ so that obviously \begin{equation*} R\Z = \bigcup_{\ell \in \{1, 2\}^{2n}} R\Z_\ell. \end{equation*} Furthermore, if $\ell \in \{1, 2\}^{2n}$ is fixed and $u, u' \in R \Z_\ell$ with $u \neq u'$ then $e_{y + u'}$ is orthogonal to $e_{x + u}$ for any $x, y \in \Z_R$.  Thus,  it is easy to see that \begin{equation*} \|D_1\|    \leq 4^n \sup_{\substack{u \in R\Z \\ \INF{u} \geq N  }} \|E_{u, z_0} ^* B E_{u, w_0}\|\end{equation*}

\noindent which means that there exists some $u_0 \in R\Z$ such that $\INF{u_0} \geq N $ and \begin{equation*}   \|E_{u_0, z_0} ^* B E_{u_0, w_0}\| \geq \frac{1}{4^{\frac{5}{2} + n} C^2} \|A\|_\Q \end{equation*}

Now note that \begin{align*} \|E_{u_0, z_0} ^* B E_{u_0, w_0}\| & = \left\|\sum_{x, y \in \Z_R } \langle B\Tk_{y + u_0 + w_0 }, \Tk_{x + u_0 + z_0 }\rangle  e_{u_0 + x} \otimes  e_{u_0 + y} \right\| \\ & = \left\|\sum_{x, y \in \Z_R } \langle B\Tk_{y + u_0 + w_0}, \Tk_{x + u_0 + z_0 }\rangle  e_{ x} \otimes  e_{ y} \right\| \\ & = \|F_{u_0 + w_0 ; R} ^* B F_{u_0 + z_0 ; R }\| \end{align*}  \noindent Finally, set $a = u_0 + w_0$ and $b =  z_0 - w_0$ so that $|a|_\infty \geq N - 1$ and $\INF{b} \leq 2$.  Then according to (\ref{XiaLem3Est1}), we have that \begin{equation*} \|F_{a ; R} ^* (A - B) F_{a + b ; R}\| \leq \frac{1}{4^{n + 3} C^2} \|A\|_\Q  \end{equation*} so that \begin{equation*} \|F_{a ; R} ^* A F_{a + b ; R}\| \geq \frac{1}{4^{n + 3} C^2}\|A\|_{\Q}  \end{equation*} which completes the proof. \end{proof}

We can now prove Theorem \ref{CompOpSuff} when $p = 2.$

\noindent \textit{Proof of Theorem \ref{CompOpSuff} when $p = 2$}.  By Lemma \ref{XiaLem3} there exists some $R \in \N$ depending on $A$ and sequences $\{a_j\}, \{b_j\} \subset \C$ with \begin{equation*} \lim_{j \rightarrow \infty} |a_j| = \infty \  \ \text{ and } \ \ \sup_{j \geq 1}  |b_j| \lesssim 2 \end{equation*} where \begin{equation*} \|A\|_{\Q} \lesssim \|F_{a_j;R}^* A F_{a_j + b_j ; R}\|. \end{equation*}

However, if $R$ is large enough, then \begin{align*} \limsup_{j \rightarrow \infty} \|F_{a_j;R}^* A F_{a_j + b_j ; R}\| & = \left\|\sum_{x, y \in \Z_R} \langle A \Tk_{x + a_j + b_j}, \Tk_{y + a_j} \rangle e_y \otimes e_x \right\| \\ & \lesssim  R^{4n} \limsup_{|z| \rightarrow \infty} \sup_{w \in B(z, 3R)} |\langle Ak_z, k_w \rangle | \end{align*}  which proves Theorem \ref{CompOpSuff} when $p = 2$.

\section{Proof of Theorem \ref{CompOpNec}} \label{SectionProofs2}

In this short section we will prove Theorem \ref{CompOpNec}. First we will need the following simple result.

\begin{lemma} \label{FockDensityLemma}  If $1 \leq p < \infty$ and $S \subseteq \C$ is a Borel set with nonzero Lebesgue volume measure, then $\s \{ K(\cdot, w) : w \in S \}$ is dense in $\Fpphi$. \end{lemma}

\begin{proof} Let $q$ be the conjugate of exponent of $p$.  If $g \in \Fqphi = (\Fpphi)^*$ (see $c)$ in Theorem \ref{FockSpacePropThm}) annihilates $\s \{ K(\cdot, w) : w \in S \}$, then $e)$ in Theorem \ref{FockSpacePropThm} implies that \begin{equation*} g(w) = \incn g(u) \overline{K(u, w)}  e^{-2\phi(u)} \, dv(u) = 0 \end{equation*} for any $w \in S$, which implies that $g \equiv 0$ since $S$ has nonzero Lebesgue volume measure.  The proof then immediately follows by the Hahn-Banach theorem.
\end{proof}

\begin{lemma} \label{CompOpNec1} Finite rank operators on $\Fpphi$ are in the norm closure of the algebra generated by Toeplitz operators with point mass measure symbols when $1 \leq p < \infty$.   \end{lemma}

\begin{proof} Since \begin{equation*} 0 < K(w, w) = \incn |K(w, z)|^2 e^{-2\phi(z)} \, dv(z) \end{equation*} the set \begin{equation*} \CalZ_w := \{z \in \C : K(w, z) \neq 0\}\end{equation*}  trivially has nonzero Lebesgue volume measure for each $w \in \C.$  Thus,  Lemma \ref{FockDensityLemma} tells us that $\s\{K(\cdot, z) : z \in \CalZ_w \}$ is dense in $\Fpphi$ for each $w \in \C$, which in turn implies that  $\s\{K(\cdot, z) \otimes K(\cdot, w) : w \in \C, z \in \CalZ_w\}$ is dense (with respect to the $\Fpphi$ operator norm) in the space of finite rank operators.

  The proof is then completed by observing that \begin{equation*} K(\cdot, z) \otimes K(\cdot, w) = \frac{e^{2\phi(z) + 2\phi(w)}} {\overline{K(w, z)}}  T_{\delta_z} T_{\delta_w} \end{equation*} where $\delta_z$ and $\delta_w$ are the point mass measures at $z, w \in \C$ with $z \in \CalZ_w$.
\end{proof}

 \begin{lemma} \label{CompOpNec2}  Given $w \in \C$, let \begin{equation*} F_w ^\epsilon (z) := \frac{c_n}{\epsilon^{2n}} \chi_{B(w, \epsilon)} (z) \end{equation*}  where $c_n$ is the volume of the unit ball in $\C$.  Then we have \begin{equation*} \lim_{\epsilon \rightarrow 0^+} \|T_{F_w ^\epsilon} - T_{\delta_w}\|_{\Fpphi \rightarrow \Fpphi} = 0 \end{equation*} for each $1 < p < \infty$.\end{lemma}

 \begin{proof} By an easy application of Theorem \ref{FockSpacePropThm} we have that $K(z, \cdot) \in \Fonephi$ for each $z \in \C$.  Thus, by Theorems \ref{FockSpacePropThm} and \ref{FockCarThm}, we have that \begin{equation*} \|T_\mu\|_{\Finfphi \rightarrow \Finfphi} \lesssim \|\mu\|_* . \end{equation*}   Therefore, by complex interpolation and duality, it is enough to prove the lemma for $p = 2$.  To that end, note that $T_{F_w ^\epsilon} - T_{\delta_w}$ is obviously bounded and self adjoint on $\Ftwophi$, which means that \begin{equation*} \|T_{F_w ^\epsilon} - T_{\delta_w}\|_{\Ftwophi \rightarrow \Ftwophi} = \sup_{\|h\|_{\Ftwophi} = 1} \left| \langle (T_{F_w ^\epsilon}  - T_{\delta_w}) h, h \rangle \right|. \end{equation*}

 However, \begin{align*} \left|\langle (T_{F_w ^\epsilon}  - T_{\delta_w}) h, h \rangle \right| & = \left| \incn |h(z)|^2 F_w ^\epsilon (z)  e^{- 2 \phi(z)} \, dv(z) - \incn |h(z)|^2   e^{- 2 \phi(z)} \, d\delta_{w} (z) \right| \\ & =  \left| \frac{c_n}{\epsilon^{2n}} \int_{B(w, \epsilon)} |h(z)|^2 e^{-2 \phi(z)} dv(z) - |h(w)|^2 e^{-2\phi(w)} \right| \\ & \leq  \frac{c_n}{\epsilon^{2n}} \int_{B(w, \epsilon)} \left||h(z)|^2 e^{- 2 \phi(z)} - |h(w)|^2 e^{- 2 \phi(w)} \right| dv(z) \end{align*} \noindent where $c_n$ is the volume of the unit ball in $\C$.
  Moreover, if $\|h\|_{\Ftwophi} = 1$, then $g)$ in Theorem \ref{FockSpacePropThm} tells us that $|h|^2 e^{- 2 \phi}$ is Lipschitz with Lipschitz constant independent of $h$, which completes the proof. \end{proof}

Note that by an easy application of Theorem $2$ in \cite{SV} we have that $T_f$ is compact on $\Fpphi$ for $1 \leq p < \infty$ if $f \in C_c^\infty(\C)$.  Combining this fact with Lemmas \ref{CompOpNec1} and \ref{CompOpNec2} gives us the following.

\begin{theorem} \label{CompOpNecPrelim} Finite rank operators are in $\Tpphi(C_c ^\infty(\C))$ when $1 < p < \infty$.    In particular, since all $L^p$ spaces have the bounded approximation property (see \cite{W}), the space of compact operators on $\Fpphi$ coincides with $\Tpphi(C_c ^\infty(\C))$.   \end{theorem}

The proof that $\{T_f : f \in C_c ^\infty(\C)\}$ is $\Fp$  operator norm dense in the space of compact operators will use Theorem \ref{CompOpNecPrelim} in conjunction with the ideas in \cite{B} p. 3136, which we now elaborate.

Note that the proof of Theorem \ref{CompOpNecPrelim} actually shows that $\s \{T_f T_g : f, g \in C_c ^\infty(\C)\}$ is $\Fp$ operator norm  dense in the space of compact operators when $1 < p < \infty$.  Thus, to show that $\{T_f : f \in C_c ^\infty(\C)\}$ is $\Fp$  operator norm dense in the space of compact operators on $\Fp$, it is enough to show that $\{T_f : f \in C_c ^\infty(\C)\}$ is $\Fp$ operator norm dense in $\s \{T_f T_g : f, g \in C_c ^\infty(\C)\}$.

To that end, Let $\F$ be the usual $L^2$ - Fourier transform on $\C$ where we identify $\C$ with $\mathbb{R}^{2n}$ in the canonical way.  Now if $f_1, f_2 \in C_c ^\infty(\C)$, then it is elementary that there exists sequences $\{f_{j, \ell}\}_{\ell = 1}^\infty \subset \F (C_c ^\infty(\C))$ for $j = 1, 2$ such that $\lim_{\ell \rightarrow \infty} f_{j, \ell} = f_j$ uniformly on $\C$, where $\F (C_c ^\infty(\C))$ is the image of $C_c ^\infty(\C)$ under $\F$. Furthermore, by Theorem $24$ in \cite{B}, we have that \begin{equation} \label{ToepCompForm1}  T_{f} T_{g} = T_{f \sharp_\alpha g} \end{equation} where the ``product" $\sharp_\alpha$ is given by
\begin{equation} \label{ToepCompForm2} f \sharp_\alpha g = \sum_{\gamma \in \mathbb{N}_0 ^n} \frac{1}{(-\alpha)^{|\gamma|}\gamma!} \frac{\partial ^{|\gamma|} f}{\partial z ^\gamma} \cdot \frac{\partial ^{|\gamma|} g}{\partial \overline{z} ^\gamma} \end{equation} whenever $f, g \in \F(C_c ^\infty(\C))$.

By uniform convergence and (\ref{ToepCompForm1}), we have that \begin{equation*} T_{f_1} T_{f_2} =  \lim_{\ell \rightarrow \infty} T_{f_{1, \ell}} T_{f_{2, \ell}} =  \lim_{\ell \rightarrow \infty} T_{f_{1, \ell} \sharp_\alpha f_{2, \ell}} \end{equation*} where the limit is in the $\Fpphi$ operator norm.  Finally, since each $f_{j, \ell}$ is in the Schwartz space of $\C$,  it is clear that each $f_{1, \ell} \sharp_\alpha f_{2, \ell}$ is smooth and vanishes at infinity.  Thus, each $f_{1, \ell} \sharp_\alpha f_{2, \ell}$ can itself be uniformly approximated on $\C$ by functions in $C_c ^\infty (\C)$, which completes the proof that $\{T_f : f \in C_c ^\infty(\C)\}$ is dense in the space of compact operators on $\Fp$.

Finally, we will complete the proof of Theorem \ref{CompOpNec} when $p = 2$.   As was stated in the introduction, the proof is very similar to the proof of Theorem $9$ in \cite{BC} and so we will only outline the proof.  To that end, given any bounded operator $X$ on $\Ftwophi$, let $K_X(w, z)$ be the function defined by \begin{equation*} K_X(w, z) := (X^* K(\cdot, z))(w) \end{equation*} so that $K_X(w, z)$ is analytic in $w$ and conjugate analytic in $z$.

 Note that $a)$ in Theorem $2.1$ immediately tells us that $P M_S : \Ltwophi \rightarrow \Ltwophi$ is Hilbert-Schmidt if $S \subseteq \C$ is compact, which easily implies that $T_f$ is trace class on $\Ftwophi$ when $f \in C_c ^\infty (\C)$.  Thus, if $f \in C_c ^\infty (\C)$ and $X$ is any bounded operator on $\Ftwophi$, then $T_f X$ is trace class on $\Ftwophi$ and repeating word for word the proof of Theorem $8$ in \cite{BC} gives us that \begin{equation}\label{TraceFormula}  \tr (T_g X) = \int_{\C} g(w) \overline{K_X (w, w)} e^{-2\phi(w)} \, dv(w). \end{equation}

Now suppose that $\left\{T_f : f \in C_c ^\infty (\C) \right\}$ is \textit{not} dense in the space of compact operators on $\Ftwophi$.  Then by the Hahn-Banach theorem and duality, there exists a non-zero trace class operator $X$ on $\Ftwophi$ where $\tr(T_g X) = 0$ for any $g \in C_c ^\infty(\C)$.  However, this implies that \begin{equation*} 0 = \int_{\C} g(w) \overline{K_X (w, w)} e^{-2\phi(w)} \, dv(w) \end{equation*} for any $g \in C_c ^\infty(\C)$, which by elementary arguments implies that $K_X(w, w) \equiv 0$.

The proof will be completed if we can show that \begin{equation*} K_X(w, w) \equiv 0 \Longrightarrow X = 0. \end{equation*}  To that end, since $K_X(w, z)$ is analytic in $w$ and conjugate analytic in $z$ and $K_X(w, w) \equiv 0$, a standard result in several complex variables implies that $K_X(w, z) \equiv 0$.  However, since $\s \{K(\cdot, z) : z \in \C\}$ is dense in $\Ftwophi$, the condition $K_X(w, z) \equiv 0$ implies that $X = 0$.  (It should be noted that the argument in this paragraph is by now standard and that the exact same argument tells us that the Berezin transform is injective on $\Fpphi$ when $1 < p < \infty$. )

It should be remarked that a very similar argument also shows that $\left\{T_f : f \in C_c ^\infty (\C) \right\}$ is trace norm dense in the trace class of $\Ftwophi$ (which was proved in \cite{BC} for the classical Fock space $\Ft$.)

\section{Essential norm estimates} \label{EssNormSec}
In this section we prove Theorems \ref{GenEssNormProp} and \ref{ToepEssNormThm}. First however we will need the following two Lemmas, the second of which is similar to Proposition $4.4$ in \cite{MW}.

\begin{lemma} \label{GenEssNormLem1} If $K$ is compact on $\Fpphi$ and $1 < p < \infty$, then \begin{equation*}  \limsup_{R \rightarrow \infty} \|M_{\chi_{B(0, R) ^c}} K   \|_{\Fpphi \rightarrow \Lpphi} = 0 \end{equation*} \end{lemma}

\begin{proof} By Theorem \ref{CompOpNec} and an easy approximation argument, it is enough to prove the result for $K = T_f$ where $f \in C_c ^\infty(\C)$.

For that matter, let $S_f = \supp f$ and let $M = \sup \{|w| : w \in S_f\}$.  Furthermore, assume without loss of generality that $R > M$.    If $g \in \Fpphi$ with $\|g\|_{\Fpphi} = 1$ and $q$ is the conjugate exponent of $p$, then we have
\begin{align*} |e^{-\phi(z)} {\chi_{B(0, R) ^c} (z)}T_f g(z)|  & \leq \chi_{B(0, R) ^c} (z) \int_{S_f}  |f(w)||g(w)| |K(z, w)| e^{-\phi(z)} e^{-2\phi(w)} \, dv(w) \\ & \leq \|f\|_{L^\infty}  \chi_{B(0, R) ^c} (z) \left(\int_{S_f} \left(e^{\phi(z)} |K(z, w)| e ^{\phi(w)}\right)^q \, dv(w) \right)^\frac{1}{q} \\ & \lesssim \|f\|_{L^\infty} \chi_{B(0, R) ^c} (z)\left(\int_{S_f} e^{-q\epsilon |z - w|} \, dv(w) \right)^\frac{1}{q} \\ & \lesssim \|f\|_{L^\infty} e^{-\frac{ \epsilon (R - M) }{2}} e^{-\frac{\epsilon(|z| - M)}{2}} \end{align*} which immediately implies that \begin{equation*} \| M_{\chi_{B(0, R) ^c}} T_f g\|_{\Fpphi \rightarrow \Lpphi} \lesssim \|f\|_{L^\infty} e^{-\frac{ \epsilon (R - M) }{2}}   \end{equation*} where $\epsilon$ is from Theorem \ref{FockSpacePropThm}.  Letting $R \rightarrow \infty$ now completes the proof.  \end{proof}

Before we prove the next lemma, we will need to use the simple covering of $\C$ from \cite{BI}.  In particular, fix $d > 0$ and enumerate the disjoint family of sets
$\{[- d, d)^{2n} + \sigma\}_{\sigma \in 2 d \mathbb{Z}^{2n}}$  as $\{F_j\}_{j =
1}^\infty$ and for this fixed $d$ let  $G_j = \{z \in \mathbb{C}^n : \dist_\infty (z, F_j) \leq d
\}$ where $\dist_\infty (z, F_j)$ is the distance between $z$ and $F_j$ in the $|\cdot|_\infty$
norm. The following properties now hold trivially from the definitions above:

 \begin{list}{}{\setlength\parsep{0in}} \item $(a) \ F_j \cap F_k = \emptyset $
if $j \neq k$, \item $(b) \ $Every $z \in \mathbb{C}^n$ belongs to at most $2^{2n}$ of the sets
$G_j$, \item $(c) \  \diam(G_j) \leq 4d \sqrt{2n} $ where $\diam(G_j)$ is the
Euclidean diameter of $G_j$. \end{list}

\begin{lemma} \label{GenEssNormLem2} If $\epsilon  > 0$ and $A \in \SLphi$, then there exists $d = d(A) > 0$ such that \begin{equation*} \left\|AP - \sum_{j} M_{\chi_{F_j}}  AP M_{\chi_{G_j}} \right\|_{\Fpphi \rightarrow \Lpphi} < \epsilon \end{equation*} \end{lemma} where the sets $F_j$ and $G_j$ are defined above.

\begin{proof}  We first prove the Lemma for $p = 2$.  To that end, note that \begin{align*} (APf) (w) &  - \sum_{j} \chi_{F_j}(w)   (AP M_{\chi_{G_j}} f)(w)  = \sum_{j} \chi_{F_j} (w)  (AP M_{\chi_{G_j ^c}} f)(w) \\ & = \incn \Phi(w, u) f(u) e^{-2\phi(u)} \, dv(u) \end{align*} where \begin{equation*} \Phi(w, u) := \sum_j {\chi_{F_j}} (w)  {\chi_{G_j ^c}} (u) \langle A K(\cdot, u), K(\cdot, w) \rangle.\end{equation*}

We then estimate that \begin{align*}
\incn |\Phi(w, u) | (e^{\phi(u)}) e^{-2\phi(u)}   \,   dv(u) & \approx e^{\phi(w)}  \sum_j \int_{G_j ^c}  \chi_{F_j } (w) |\langle A k_u, k_w \rangle| \, dv(u)
 \\ & \lesssim e^{\phi(w)}  \sum_j \int_{G_j ^c} \frac{\chi_{F_j}  (w) }{(1 + |u - w|)^{2n + \delta}} \, dv(u) \\ & \lesssim d^{-\frac{\delta}{2}} e^{\phi(w)} \end{align*} since $|u - w| \gtrsim d$ if $u \in F_j$ and $w \in G_j ^c$.  Similarly we can easily get that \begin{equation*} \incn |\Phi(w, u) | (e^{\phi(w)}) e^{-2\phi(w)}   \, dv(w) \lesssim d^{-\frac{\delta}{2}} e^{\phi(u)} \end{equation*} which by the Schur test proves the lemma if $p = 2$.

Now assume that $1 < p < 2$.  Since $A$ is bounded on $\Fonephi$ we easily get that \begin{equation*} \left\|\sum_{j} M_{\chi_{F_j}}  AP M_{\chi_{G_j}} \right\|_{\Fonephi \rightarrow \Lonephi} < \infty \end{equation*} which by complex interpolation proves the Theorem when $1 < p < 2$.

Finally when $2 < p < \infty$, one can similarly get a trivial $\Lonephi \rightarrow \Fonephi$ operator norm bound on \begin{equation*}\left( \sum_{j} M_{\chi_{F_j}}  AP M_{\chi_{G_j}} \right) ^* =  \sum_j P M_{\chi_{G_j}} A^* P M_{\chi_{F_j}} \end{equation*} since $A^*$ is bounded on $\Fonephi$. Duality and complex interpolation now proves the lemma when $2 < p < \infty$.

\end{proof}

We will now prove Theorem \ref{GenEssNormProp}

\textit{Proof of Theorem \ref{GenEssNormProp}.} We first prove (\ref{GenEssNormEst1}).  Let $A$ be
bounded on $\Fpphi$. Then since $P M_{\chi_{B(0, R)}} A$ is compact on $\Fpphi$ for any $R > 0$ and since the orthogonal projection $P : \Lpphi \rightarrow \Fpphi$ is bounded and coincides with the identity on $\Fpphi$, we have that \begin{equation*} \|A\|_\Q \leq \limsup_{R \rightarrow \infty} \|P A -  P M_{\chi_{B(0, R)}} A \|_{\Fpphi \rightarrow \Fpphi} \lesssim \limsup_{R \rightarrow \infty} \| M_{\chi_{B(0, R) ^c}} A \|_{\Fpphi \rightarrow \Lpphi}. \end{equation*}

On the other hand, if $K : \Fpphi \rightarrow \Fpphi$ is compact then Lemma \ref{GenEssNormLem1}  gives us that \begin{align*} \limsup_{R \rightarrow \infty} \| M_{\chi_{B(0, R) ^c}} A \|_{\Fpphi \rightarrow \Lpphi}  & = \limsup_{R \rightarrow \infty} \| M_{\chi_{B(0, R) ^c}}  (A - K) \|_{\Fpphi \rightarrow \Lpphi} \\ & \leq \|A - K \|_{\Fpphi \rightarrow \Fpphi} \end{align*} which completes the proof of (\ref{GenEssNormEst1}).

Now we will prove (\ref{GenEssNormEst2}).  By completely elementary arguments we have that \begin{equation*} \sup_{d > 0} \limsup_{|z| \rightarrow \infty} \|M_{\chi_{B(z, d)}} A P M_{\chi_{B(z, 2d)}}\|_{\Fpphi \rightarrow \Lpphi} \leq \limsup_{R \rightarrow \infty} \| M_{\chi_{B(0, R) ^c}} A \|_{\Fpphi \rightarrow \Lpphi} \end{equation*} for any bounded $A$ on $\Fpphi$.  Finally, since \begin{equation*}  \|A\|_Q \approx  \limsup_{R \rightarrow \infty} \| M_{\chi_{B(0, R) ^c}} A \|_{\Fpphi \rightarrow \Lpphi}, \end{equation*} we will complete the proof by showing that \begin{equation*} \|A\|_{\Q} \lesssim \sup_{d > 0} \limsup_{|z| \rightarrow \infty} \|M_{\chi_{B(z, d)}} AP M_{\chi_{B(z, 2d)}} \|_{\Fpphi \rightarrow \Lpphi} \end{equation*}  for any $A \in \SLphi$.  An easy approximation argument will then complete the proof.

To that end, let $\epsilon > 0$. Fix some $d > 0$ large enough where Lemma \ref{GenEssNormLem2} is true and let $\{F_j\}$ be the corresponding cover of $\C$ (with associated sets $\{G_j\}$.)  Then we have that \begin{equation*} \|A\|_\Q \leq \epsilon + \limsup_{m \rightarrow \infty} \left\| \sum_{j \geq m}  M_{\chi_{F_j}}  AP M_{\chi_{G_j}}  \right\|_{\Fpphi \rightarrow \Lpphi} \end{equation*}  However, if $f \in \Fpphi$ with norm one, then \begin{align*} \limsup_{m \rightarrow \infty} \left\| \sum_{j \geq m} M_{\chi_{F_j}}  AP M_{\chi_{G_j}}  f \right\|_{\Lpphi} ^p & = \limsup_{m \rightarrow \infty} \sum_{j\geq m} \left\| M_{\chi_{F_j}}  AP M_{\chi_{G_j}}  f \right\|_{\Lpphi} ^p \\ & \leq 2^{2n} \limsup_{m \rightarrow \infty} \|M_{\chi_{F_m}}  AP M_{\chi_{G_m}} \|_{\Fpphi \rightarrow \Lpphi} ^p \\ & \leq 2^{2n} \sup_{d > 0} \limsup_{|z| \rightarrow \infty} \|M_{\chi_{B(z, d)}} A  P M_{\chi_{B(z, 2d)}} \|_{\Fpphi \rightarrow \Lpphi} ^p.  \end{align*}  Letting $\epsilon \rightarrow 0^+$ now completes the proof.

We will now prove an extremely useful technical lemma whose proof is similar to the proof of Lemma \ref{GenEssNormProp} and part $(a)$ of Theorem $4.3$ in \cite{MW}. For the sake of notational ease, all norms in the rest of this section will either denote the $\Fpphi$ norm, the $\Fpphi$ operator norm, or the $\Fpphi \rightarrow \Lpphi$ norm.

\begin{lemma}  \label{MWLem} Suppose that $1 < p < \infty$ and let $\epsilon > 0$.  Pick $d > 0$ corresponding to $\epsilon $    in Lemma \ref{GenEssNormLem2}.  Then there exists a sequence  $\{z_j\}$ with $\lim_{j \rightarrow \infty} |z_j| = \infty$ such that  \begin{equation*}  \|A\|_Q \leq \epsilon +  \limsup_{j \rightarrow \infty} \|M_{\chi_{B(z_j, d \sqrt{2n})}} Ag_j \| \end{equation*} where \begin{equation*} g_j  :=  \int_{B(0, 2d\sqrt{2n})} a_j(u) \Tk_{z_j - u} \, dv(u) \end{equation*} and where $a_j$ satisfies    \begin{equation*} \int_{B(0, 2d\sqrt{2n})} |a_j(u)|^p \, dv(u) = 1. \end{equation*} \end{lemma}

\begin{proof} As in the proof of Theorem \ref{GenEssNormProp},  fix $d > 0$ such that \begin{equation*}  \|A\|_{\Q} \leq  \frac{\epsilon}{2} + \limsup_{m \rightarrow \infty} \left\|\sum_{j \geq m} M_{\chi_{F_j}} AP M_{\chi_{G_j}}\right\|.  \end{equation*}   However, if $ \|f\| \leq 1$,  then  \begin{align*} \left\|\sum_{j \geq m} M_{\chi_{F_j}} AP M_{\chi_{G_j}} f \right\|  ^p  & = \sum_{j \geq m} \left\| M_{\chi_{F_j}} AP M_{\chi_{G_j}} f \right\|  ^p \\ & = \sum_{j \geq m} \frac{\left\| M_{\chi_{F_j}} AP M_{\chi_{G_j}} f\right\|  ^p }{\|M_{\chi_{G_j}} f \|^p}   \|M_{\chi_{G_j}} f \|^p \\ & \leq 2^{2n} \sup_{j \geq m} \|M_{\chi_{F_j}} A l_j \|^p \end{align*} where \begin{equation*} l_j := \frac{ P M_{\chi_{G_j}}f}{\| M_{\chi_{G_j}}f\|}.\end{equation*} If $w_j$ is the center of the cubes $F_j$ then $F_j \subset B(w_j, d\sqrt{2n}) $ so that  $G_j \subset B(w_j, 2d\sqrt{2n})$. Now if \begin{equation*} T_m := \sum_{j \geq m} M_{\chi_{F_j}} AP M_{\chi_{G_j}} \end{equation*} then we have that \begin{align*} \|T_m\| & \lesssim \sup_{j \geq m} \sup_{\|f\| \leq 1} \left\{ \|M_{\chi_{F_j}}A l_j \| : l_j =  \frac{P M_{\chi_{G_j}} f }{\| M_{\chi_{G_j}} f \|} \right\} \\ & \lesssim   \sup_{|z| \geq |w_m|} \sup_{\|f\| \leq 1} \left\{ \|M_{B(z, d\sqrt{2n})}A g \| : g = \frac{P M_{B(z, 2d\sqrt{2n})} f }{\| M_{B(z, 2d\sqrt{2n})}f \|} \right\} \end{align*} and so \begin{equation*} \limsup_{m \rightarrow \infty} \|T_m\| \lesssim \limsup_{|z| \rightarrow \infty} \sup_{\|f\| \leq 1} \left\{ \|M_{B(z, d\sqrt{2n})}A g \| : g = \frac{P M_{B(z, 2d\sqrt{2n})} f }{\| M_{B(z, 2d\sqrt{2n})} f \|} \right\}. \end{equation*}

  Pick a sequence $\{z_j\} \subset \C$ and a corresponding sequence $\{f_j\} \subset \Fpphi$ with $\|f_j \| \leq 1$ such that \begin{align*} \limsup_{|z| \rightarrow \infty} & \sup_{\|f\| \leq 1} \left\{ \|M_{B(z, d\sqrt{2n})}A g \| : g = \frac{P M_{B(z, d\sqrt{2n})} f }{\| M_{B(z, d\sqrt{2n})} f \|} \right\} - \frac{1}{2} \epsilon \\ & \leq \limsup_{j \rightarrow \infty} \|M_{B(z_j, d\sqrt{2n})}Ag_j \| \end{align*} where \begin{align*} g_j  & :=  \frac{P M_{B(z_j, 2d\sqrt{2n})} f_j }{\| M_{B(z_j, 2d\sqrt{2n})}f_j \|}   = \frac{\int_{B(z_j, 2d\sqrt{2n})} \langle f_j, \Tk_u \rangle \Tk_u \, dv(u)}{\left(\int_{B(z_j, 2d\sqrt{2n})} |\langle f_j, \Tk_w\rangle|^p \, dv(w) \right)^\frac{1}{p} }  \\ & = \frac{\int_{B(0, 2d\sqrt{2n})} \langle f_j, \Tk_{ z_j - u} \rangle \Tk_{ z_j - u} \, dv(u)}{\left(\int_{B(0, 2d\sqrt{2n})} |\langle f_j, \Tk_{ z_j - w} \rangle|^p \, dv(w) \right)^\frac{1}{p} } \label{ImportantEquality} \end{align*} (where the second to last equality follows from the definition of $P$, the definition of $\Tk_w$, and the reproducing property.)

 Finally, setting \begin{equation*} a_j(u) := \frac{\langle f_j, \Tk_{ z_j - u} \rangle}{\left(\int_{B(0, 2d\sqrt{2n})} |\langle f_j, \Tk_{ z_j - w} \rangle|^p \, dv(w) \right)^\frac{1}{p} } \end{equation*} completes the proof. \end{proof}

 We will now prove three very interesting corollaries to Lemma \ref{MWLem}, the first of which is a proof of Theorem \ref{CompOpSuff} when $p \neq 2$.

\noindent \textit{Proof of Theorem \ref{CompOpSuff} when $p \neq 2$.} Let $A \in \SLphi$.   We in fact prove that there exists $R > 0$ such that \begin{equation*} \|A\|_\Q \lesssim R^{2n} \limsup_{|z| \rightarrow \infty} \sup_{w \in B(z, R)} |\langle A k_z, k_w\rangle|. \end{equation*}
Obviously there is nothing to prove if $\|A\|_{\Q} = 0$ so assume $\|A\|_{\Q} > 0$.  Then by Lemma \ref{MWLem} with $\epsilon = \frac{1}{2} \|A\|_{\Q}$ we have a sequence $\{z_j\}$ with $\lim_{j \rightarrow \infty} |z_j| = \infty$ where  \begin{equation*}  \|A\|_Q \leq 2  \limsup_{j \rightarrow \infty} \|M_{B(z_j, R/2)} Ag_j \| \end{equation*} with \begin{equation*} g_j  :=  \int_{B(0, R)} a_j(u) \Tk_{ z_j - u} \, dv(u) \end{equation*}  where $a_j$ satisfies \begin{equation*} \int_{B(0, R)} |a_j(u)|^p \, dv(u) = 1 \end{equation*} and where $R := 2d\sqrt{2n}$ with $d$ coming from Lemma \ref{MWLem}. However, the reproducing property gives us that \begin{equation*} |A g_j (z) | e^{-\phi(z)} \leq \int_{B(0, R)} |a_j(u)|  |\langle A \Tk_{ z_j - u}, \Tk_{z} \rangle| \, dv(u) \end{equation*} so that by H\"{o}lder's inequality we have \begin{align*} \|A\|_Q ^p  \leq & 2^p \limsup_{j \rightarrow \infty} \int_{B(z_j, R)} \left( \int_{B(0, R)} |a_j(u)|  |\langle A \Tk_{z_j - u}, \Tk_{z} \rangle| \, dv(u) \right)^p \, dv(z) \\ & \lesssim R^{2np} \limsup_{|z| \rightarrow \infty} \sup_{w \in B(z, 2R)} |\langle A k_z, k_w\rangle| ^p \end{align*} which completes the proof.

We will now prove Theorem \ref{ToepEssNormThm} with the help of Lemma \ref{MWLem}.

 \noindent \textit{Proof of Theorem \ref{ToepEssNormThm}.}  First note that $\|T_\mu\| \lesssim 1$ if $\|\mu\|_* \leq 1$ so without loss of generality we can assume that $0 < \|T_\mu\|_\Q  <  1$ since otherwise there is nothing to prove. By Lemma \ref{MWLem} there exists a sequence  $\lim_{j \rightarrow \infty} |z_j| = \infty$ where  \begin{equation*}  \|A\|_Q \leq 2 \limsup_{j \rightarrow \infty} \|Ag_j \| \end{equation*} where \begin{equation*} g_j  :=  \int_{B(0, 2d\sqrt{n})} a_j(u) \Tk_{ z_j - u} \, dv(u) \end{equation*}   and where \begin{equation*} \int_{B(0, 2d\sqrt{n})} |a_j(u)|^p \, dv(u) = 1. \end{equation*} However, by the proofs of Proposition \ref{ToepOpInSL} and Lemma \ref{GenEssNormLem2}, we can pick $d > 0$ where $e^{-\frac{\epsilon d}{2}} \lesssim \|T_\mu\|_Q$ (where here $\epsilon$ corresponds to Proposition \ref{FockSpacePropThm}) so without loss of generality we may assume that $d = - \ln (\|T_\mu \|_Q )$

Furthermore, combining this with  the proof of Theorem \ref{CompOpSuff} when $p \neq 2$, we have that \begin{align*} \|T_\mu\|_\Q & \lesssim \left(\ln (\|T_\mu \|_Q )\right)^{{2n}} \limsup_{|z| \rightarrow \infty} \sup_{w \in B(z, R)} |\langle A k_z, k_w\rangle| \\ & \leq \left(\ln (\|T_\mu \|_Q )\right)^{{2n}} \limsup_{|z|, |w|  \rightarrow \infty} |\langle A k_z, k_w\rangle|  \end{align*} (where as before $R := 2d \sqrt{2n}$.)

Finally,  it is elementary that $u/(\ln u)^{2n} \geq C_\delta u ^\frac{1}{\delta}$ for $u \in (0, 1)$ which means that \begin{equation*} (\|T_\mu\|_{\Q})^\frac{1}{\delta} \leq C_\delta \limsup_{|z|, |w| \rightarrow \infty} |\langle T_\mu k_z, k_w\rangle |  \end{equation*} for all $0 < \delta < 1. $

We will end this section by extending the main essential norm estimate in \cite{MW} to the $\Fpphi$ setting in the situation where we are assuming the existence of a uniformly bounded family of operators $\{U_z\}_{z \in \C}$ on $\Fpphi$ such that (\ref{UzDef}) is true and where $\|U_z h\|_{\Fpphi} \gtrsim \|h\|_{\Fpphi}$ for all $z \in \C$ and $h \in \Fpphi$.  In particular, we will prove the following result whose proof is similar to the proof of part $(a)$ of Theorem $4.3$ in \cite{MW} .  It should be remarked that this provides a vastly simplified proof of the main results in \cite{BI} when $p \neq 2$. Also note that this theorem should be interpreted as another way of quantifying the statement that $\|A\|_{\Q}$ is equivalent to the ``norm of $A$ translated out to infinity."

\begin{theorem} \label{MWThm} Assuming the existence of a uniformly bounded family of operators $\{U_z\}_{z \in \C}$ on $\Fpphi$ satisfying (\ref{UzDef}) and where $\|U_z h\|_{\Fpphi} \gtrsim \|h\|_{\Fpphi}$ for all $z \in \C$ and $h \in \Fpphi,$ we have that \begin{eqnarray*} \|A\|_{\Q} \approx \sup_{\|f\|_{\Fpphi} \leq 1} \limsup_{|z| \rightarrow \infty} \| A U_z f \|_{\Fpphi } \end{eqnarray*} holds for any $A$ in the $\Fpphi$ operator norm closure of $\SLphi$ when $1 < p < \infty.$  \end{theorem}

\begin{proof}
Let $w \in \C$ and notice that \begin{equation*} \limsup_{|z| \rightarrow \infty} \| K U_z k_w \|_{\Fpphi} \lesssim \limsup_{|z| \rightarrow \infty} \| K k_{ z - w} \|_{\Fpphi} = 0  \end{equation*} if $K$ is compact on $\Fpphi$.   Thus, by an easy density argument, we have that \begin{equation*} \sup_{\|f\|_{\Fpphi} \leq 1} \limsup_{|z| \rightarrow \infty} \| K U_z f \|_{\Fpphi } = 0  \end{equation*} if $K$ is compact on $\Fpphi$.  In particular, if $K$ is compact on $\Fpphi$ then \begin{equation*}  \sup_{\|f\|_{\Fpphi} \leq 1} \limsup_{|z| \rightarrow \infty} \| A U_z f \|_{\Fpphi}  = \sup_{\|f\|_{\Fpphi} \leq 1} \limsup_{|z| \rightarrow \infty} \| (A - K) U_z f \|_{\Fpphi}  \lesssim  \| A - K  \|_{\Fpphi \rightarrow \Fpphi} \end{equation*} so that  \begin{align*}  \sup_{\|f\|_{\Fpphi} \leq 1} \limsup_{|z| \rightarrow \infty} \| A U_z f \|_{\Fpphi} \lesssim \|A\|_{\Q}. \end{align*}

Now for the other half of Theorem \ref{MWThm}, let $\epsilon > 0$ and pick a sequence  $\{z_j\}$ with $\lim_{j \rightarrow \infty} |z_j| = \infty$ where  \begin{equation*}  \|A\|_Q \leq \epsilon +  \limsup_{j \rightarrow \infty} \|Ag_j \|_{\Fpphi} \end{equation*} and where \begin{equation*} g_j  :=  \int_{B(0, 2d\sqrt{n})} a_j(u) \Tk_{z_j - u} \, dv(u). \end{equation*}

Now let $\rho := 2d \sqrt{2n}$.  Note that we can write $\Tk_{z_j - u} = \Theta(u, z_j) U_{z_j } \Tk_u$ where $|\Theta(\cdot, \cdot)|$ is bounded above and below  on $\C \times \C$.  Thus, it is not difficult to see that we can write $g_j$ as $g_j = U_{z_j} h_j$ where \begin{equation*} h_j =  \int_{B(0, \rho)} a_j(u) \Tk_u \, dv(u)
\end{equation*}  and where  \begin{equation*} a_j(u) := \frac{ \Theta(u, z_j) \langle f_j, \Tk_{z_j - u}\rangle}{\left(\int_{B(0, \rho)} |\langle f_j, \Tk_{z_j - u} \rangle|^p \, dv(u) \right)^\frac{1}{p}} \end{equation*}

Since $\{a_j\}$ is a bounded sequence in $L^p(B(0, \rho), dv)$, (passing to a subsequence if necessary) we can assume that $a_j $ converges in the weak $-*$ topology of $L^q(B(0, \rho))$ to a function $a$ on $B(0, \rho)$ (where $q$ is the conjugate exponent of $p$), which in particular means that we may also assume $a_j \rightarrow a$ pointwise on $B(0, \rho)$. Now if \begin{equation*} h =  \int_{B(0, \rho)} a(u) \Tk_u \, dv(u)
\end{equation*} then an easy application of the Lebesgue dominated convergence theorem (in conjunction with Theorem \ref{FockSpacePropThm}) gives us that $h_j \rightarrow h$ in $\Fpphi$.  Moreover, we have that \begin{equation*} 1 \gtrsim \|g_j\|_{\Fpphi} = \|U_{z_j} h_j \|_{\Fpphi} \approx \|h_j\|_{\Fpphi} \end{equation*} so that $\|h\| \lesssim 1$, and finally this gives us that \begin{equation*}   \|A\|_{\Q} \lesssim \epsilon  + \limsup_{j \rightarrow \infty} \|Ag_j \|  = \epsilon + \limsup_{j \rightarrow \infty} \|A U_{z_j} h_j \|  \lesssim \epsilon + \limsup_{j \rightarrow \infty} \|A U_{z_j} h \|  \end{equation*} and hence \begin{equation*} \|A\|_{\Q} \lesssim \sup_{\|f\| \leq 1} \limsup_{|z| \rightarrow \infty} \| AU_{z_j} f\| \end{equation*}  \end{proof}

\section{Open problems \label{OpenProblems}}
In this last section we will discuss some interesting open problems related to the results of this paper. The first obvious question is whether Theorem \ref{CompOpSuff} holds where we replace (\ref{CompOpSuffCond}) with the condition that \begin{equation} \label{BerTransVan} \lim_{|z| \rightarrow \infty} (B(A))(z) = 0\end{equation} when we do not necessarily assume the existence of a uniformly bounded family of operators $\{U_z\}_{z \in \C}$  that satisfies (\ref{UzDef}). Furthermore, it would be fascinating to know whether there is any kind of converse to Proposition \ref{BerProp} in the following sense: suppose that $\Ftwophi$ (or more generally $\Fpphi$) satisfies the condition that $(\ref{BerTransVan}) \Longrightarrow  (\ref{CompOpSuffCond})$ for all $R > 0$ and all bounded operators $A$ on $\Ftwophi$ (respectively, $\Fpphi$.)  Then does this necessarily imply the existence of a uniformly bounded family of operators $\{U_z\}_{z \in \C}$ satisfying (\ref{UzDef})?

  Now assume that reproducing kernels of $\Ftwophi$ satisfy \begin{equation} \label{FockSpaceAssumption} |\langle k_z, k_w\rangle| \approx \frac{1}{\|K(\cdot, z - w)\|_{\Ftwophi}} \end{equation} (which in fact is assumed in \cite{MW} and is true for the classical Fock space and in an appropriately modified form is true for the classical Bergman spaces over bounded symmetric domains.) Then a simple computation tells us that \begin{equation*}  U_z ^* f (w) := f(z - w) k_z(w) \end{equation*} defines a uniformly bounded family of operators on $\Fpphi$ such that (\ref{UzDef}) holds.

  Moreover, if (\ref{FockSpaceAssumption}) is true, then it is very easy to show that \textit{any} bounded operator $U_z$ on $\Fpphi$ satisfying (\ref{UzDef}) must be defined by \begin{equation}  \label{UzDef2} U_z ^* f (w) := C(z) f(z - w) k_z(w) \end{equation} for some function $C$ on $\C$ that is bounded above and below. In particular, an easy computation tells us that we must have \begin{equation*} U_z ^* k_u (w) = \frac{\overline{\Theta(z, w)} k_u (z - w) k_z(w)}{\langle k_z, k_w\rangle\|K(\cdot, z - w)\|_{\Ftwophi}}  \end{equation*} in order for (\ref{UzDef}) to be true.  Thus, by Liouville's theorem, we have that \begin{equation*} \overline{\Theta(z, w)} = C(z) \langle k_z, k_w\rangle \|K(\cdot, z - w)\|_{\Ftwophi} \end{equation*}  for $C(\cdot)$ bounded above and below on $\C$ (since (\ref{FockSpaceAssumption}) implies that $k_z(w) \neq 0$ for all $z, w \in \C$.)  The density of the reproducing kernels on $\Fpphi$ easily completes the proof.  It is therefore reasonable to ask when in general (\ref{UzDef2}) defines a bounded (or even a well defined) operator on $\Fpphi$ and if so whether any uniformly bounded family of operators $\{U_z\}_{z \in \C}$ satisfying (\ref{UzDef}) must in fact be of the form (\ref{UzDef2}).

Now let $\widetilde{A}$ denote the Berezin transform of a bounded operator $A$ on the unweighted Bergman space $\Ltwoa$. Note that it was shown in \cite{E1} that there is no $C > 0$ independent of $f$ satisfying $\|f\|_{L^\infty(\D)} \leq 1$  where \begin{equation*} \|T_f\|_{\Q} \leq C \limsup_{|z| \rightarrow 1^{-}} |\widetilde{T_f}(z)|. \end{equation*} While it is most likely also the case that the previous statement holds true in the Fock space setting $\Ft$ (though no immediate examples come to mind), and while it is very likely that one can prove an appropriate version of Theorem \ref{ToepEssNormThm} in the $\Ltwoa$ setting, it would be interesting to know if one can set $\delta = 0$ in Theorem \ref{ToepEssNormThm}. Furthermore, it would be very interesting to know if there exists $C > 0$ independent of $\mu$ with $\|\mu\|_* \leq 1$ where \begin{equation*} \|T_\mu\|_{\Q} \leq C \limsup_{|z| \rightarrow \infty} \|T_\mu k_z\|_{\Fpphi} \end{equation*} or even if the above estimate holds true in the $\Ft$ setting for all bounded $f$ on $\C$ with $\|f\|_{L^\infty(\C)} \leq 1$.  It would also be interesting to know whether an appropriately modified estimate holds in the $\Ltwoa$ setting holds.

Now if $f \in L^q(\C, dv)$ for $1 \leq q < \infty$, then H\"{o}lder's inequality immediately implies that \begin{equation*} \|f\|_* \leq \|f\|_{L^q(\C, dv)} \end{equation*} which means that $T_f$ can be approximated in the $\Fpphi$ operator norm for $1 \leq p < \infty$ by a Toeplitz operator with $C_c ^\infty(\C)$ symbol.  Obviously this result is \textit{not} true for $f \in L^\infty(\C)$ since otherwise if $f \equiv 1$ then $T_f = \text{Id}_{\Fpphi \rightarrow \Fpphi}$ would be compact on $\Fpphi$.  However, one can ask if $T_f$ for $f \in L^\infty(\C)$ can be approximated in the $\Fpphi$ norm by Toeplitz operators with smooth, bounded symbols whose derivatives of arbitrary order are also bounded.

Note that this is in fact true for the classical Fock space $\Fp$.  In particular,  if $\mu$ is a complex Borel measure on $\C$ where $|\mu|$ is Fock-Carleson, then it was proved in \cite{BI} that for $1 < p < \infty$,   \begin{align} \label{ToepApprox} \lim_{t \rightarrow 0^+} \|T_{\widetilde{\mu}^{(t)}} - T_\mu\|_{\Fp} = 0 \end{align} where $\widetilde{\mu}^{(t)}$ is the heat transform of $\mu$ given by  \begin{equation*} \widetilde{\mu}^{(t)} (z) := \frac{1}{(4\pi t)^n} \int_{\C} e^{- \frac{|z - w|^2}{4t}} \, d\mu(w). \end{equation*}       Unfortunately the arguments used in \cite{BI} (which are similar to some of the arguments in \cite{NZZ}) are not available in the generalized Fock space setting, and therefore it would be interesting to know if the above mentioned result is true for $\Fpphi$ (even for function symbols $f \in L^\infty(\C)$.)  Note that this, if proved,  would obviously imply that the Toeplitz algebra generated by Toeplitz operators with Fock-Carleson measure symbols would coincide with the Toeplitz algebra generated by Toeplitz operators with smooth, bounded (function) symbols whose derivatives of all orders is bounded, which as was stated in Theorem $\ref{BIThm}$ is true for $\Fp$.

A related question is whether the $\Fpphi$ operator norm closure of $\SLphi$ coincides with $\Tpphi(X)$ for some class of Borel measures $X$ on $\C$ (like say, bounded functions on $\C$.)  Even in the classical Fock space setting $\Ft$ it would be interesting to know if the $\Ft$ operator norm closure of $\SLphi$ with $\phi = \alpha |\cdot|^2/2$ coincides with the Toeplitz algebra $\Tt (L^\infty(\C))$.
One rather interesting approach to this question would be to construct a ``k-Berezin transform" for $\Ft$ that is analogous to the k-Berezin transform introduced by D. Su\'{a}rez in \cite{S1} and further studied in \cite{S2,NZZ}.

It would also be interesting to know if $ \{T_f : f \in C_c ^\infty(\C)\}$ is dense in the space of compact operators on $\Fpphi$ for $1 < p < \infty$ (and $p \neq 2$.)   As was already remarked, the arguments in Section \ref{SectionProofs2} actually show that $\s \{T_f T_g : f, g \in C_c ^\infty(\C)\}$ is dense in $\Fpphi$ when $1 < p < \infty$, which is ``not too far" from $ \{T_f : f \in C_c ^\infty(\C)\}.$ Furthermore, note that the formulas (\ref{ToepCompForm1}) and (\ref{ToepCompForm2}) hold for symbols other than those in the space $\F(C_c ^\infty (\C))$ (see \cite{B} for more details.)  However, (\ref{ToepCompForm1}) and (\ref{ToepCompForm2}) are most emphatically exclusive to the classical Fock space setting.  In particular, since Toeplitz operators with ``nice" function symbols on $\Ft$ are unitarily equivalent to certain Weyl $\Psi$DOs on $L^2(\mathbb{R}^n)$ under the Bargmann isometry $\mathcal{B} :  \Ft \rightarrow L^2(\mathbb{R}^n)$, one can informally use the well known asymptoptic composition formula for the product of $\Psi$DOs and pull back to $\Ft$ to guess (\ref{ToepCompForm1}) and (\ref{ToepCompForm2}) (see \cite{F} or \cite{Zhu} for a much more detailed description of the above ideas.)  Because of this, it will most likely require new techniques to prove that $ \{T_f : f \in C_c ^\infty(\C)\}$ is dense in the space of compact operators on $\Fpphi$ for general $1 < p < \infty$.

Finally, we end this paper with a simple but nonetheless interesting fact.  First, note that the argument used to prove Theorem 1' in \cite{E} extends to the generalized Fock space setting and shows that $\{T_f : f \in C_c ^\infty(\C)\}$ is $\SOT$ dense in the space of bounded operators on $\Ftwophi$.  In particular, suppose that $f_1, \ldots, f_k$ and $g_1, \ldots, g_m$ for $k, m \in \N$ are two sequences of linearly independent functions in $\Ftwophi$.  If we now define $R : C_c ^\infty (\C) \rightarrow \mathbb{C}^{k \times m}$ by \begin{equation*} (R_\phi)_{ij} := \int_{\C} \phi(z) f_i (z) \overline{g_j(z)} \, e^{-2\phi(z)} \, dv(z) \end{equation*} then it is not difficult to show that $R$ is surjective, which by elementary Hilbert space arguments proves the claim.

\end{document}